\documentclass{amsart}

\usepackage{amsmath,amssymb,amsthm,amsfonts,graphicx}

\graphicspath{{../../drawings/}}
\usepackage[dvipsnames]{xcolor}
\definecolor{limegreen}{rgb}{0.196,0.804,0.196}
\definecolor{darkgreen}{rgb}{0.0,0.5,0.0}
\definecolor{darkbluegreen}{rgb}{0,0.3,0.6}
\definecolor{badgerred}{rgb}{0.715,0.004,0.004}
\usepackage[font=small,labelfont={bf,sf}, textfont={sf},margin=0em]{caption}
\usepackage{url,hyperref}
\hypersetup{colorlinks,
citecolor=badgerred,
filecolor=black,
linkcolor=blue,
urlcolor=darkgreen,
bookmarksopen=false,
pdftex}

\newcommand{\bd}{{\bar d}}
\newcommand{\bu}{{\bar u}}
\newcommand{\bv}{{\bar v}}

\newcommand{\cH}{\mathcal{H}}

\newcommand{\cC}{\mathcal{C}}

\newcommand{\cL}{\mathcal{L}}

\newcommand{\cN}{\mathcal{N}}
\newcommand{\cP}{\mathcal{P}}

\newcommand{\cI}{\mathcal{I}}

\newcommand{\Rm}{\mathrm{Rm}}
\newcommand{\ric}{\mathrm{Ric}}

\newcommand{\pd}{\partial}
\newcommand{\pdd}[1]{\frac\partial{\partial #1}}

\newcommand{\bdt}{{ \sigma(\tau)}}

\newcommand{\R}{{\mathbb R}}

\newcommand{\sk}{{\smallskip}}

\newcommand{\ds}{\displaystyle}
\newtheorem{theorem}{Theorem}[section]
\newtheorem{proposition}[theorem]{Proposition}
\newtheorem{lemma}[theorem]{Lemma}
\newtheorem{definition}[theorem]{Definition}
\newtheorem{prop}[theorem]{Proposition}
\newtheorem{corollary}[theorem]{Corollary}

\newtheorem{conjecture}[theorem]{Conjecture}
\theoremstyle{remark}

\newtheorem{remark}[theorem]{Remark}

\numberwithin{equation}{section}
\numberwithin{theorem}{section}

\begin{document}
\title[Asymptotics (\today)]
{Unique asymptotics of ancient compact non-collapsed solutions
to the\\
3-dimensional Ricci flow} \date{\today}
\author[Angenent]{Sigurd Angenent}
\address{Department of Mathematics, University of Wisconsin -- Madison}
\author[Daskalopoulos]{Panagiota Daskalopoulos}
\address{Department of Mathematics, Columbia University, New York}
\author[Sesum]{Natasa Sesum}
\address{Department of Mathematics, Rutgers University, New Jersey}

\thanks{
P.
Daskalopoulos thanks the NSF for support in DMS-1266172.
N.
Sesum thanks the NSF for support in DMS-1056387.
}

\begin{abstract}
  We consider compact noncollapsed ancient solutions to the 3-dimensional Ricci flow
  that are rotationally and reflection symmetric.  We prove that these solutions are either the spheres or they all have unique asymptotic behavior as $t\to -\infty$ and we give their precise asymptotic description.  This description applies in particular to the solution constructed by G.  Perelman in \cite{Pe1}.
\end{abstract}
\maketitle
\tableofcontents

\section{Introduction}
A solution to a geometric evolution equation such as the Ricci flow, the Mean Curvature Flow or the Yamabe flow is called \emph{ancient} if it exists for all $t\in(-\infty, t_0]$, for some $t_0 \leq +\infty$.  In the special case where the ancient solution
exists for all $t \in (-\infty, +\infty)$, it is called {\em eternal}.
For all these flows, the requirement that a solution should exist for \emph{all} times $t\leq t_0$, combined with some type of {\em positive curvature} condition or a suitable {\em non-collapsedness} assumption, turns out to be quite restrictive.  In a number of cases there are results which state that the list of possible ancient solutions to a given geometric flow consists of self similar solutions (``solitons'') and a shorter list of non self similar solutions.  Such classification
results play an important role in the singularity analysis of the flow, as blow up limits at a singularity give rise to an ancient solution.

\sk

Ancient compact solutions to the \emph{2-dimensional Ricci flow} were classified by Daskalopoulos, Hamilton and Sesum in \cite{DHS}.
It turns out that in this case, the complete list contains (up to conformal invariance) only the shrinking sphere solitons and the King solution.
The latter is a well known example of ancient {\em collapsed} Ricci flow solution and is given by an explicit formula.
It was first discovered by J.  King \cite{K1} in the context of the logarithmic fast-diffusion equation on $\R^2$ and later independently by P.  Rosenau \cite{R} in the same context.
It also appears as the {\em sausage model} in the context of quantum field theory, in the independent work of Fateev-Onofri-Zamolodchikov \cite{FOZ}.
Although the King ancient solution is not a soliton, it may be visualized as two steady solitons, called ``cigars'', coming from opposite spatial infinities, glued together.
Let us remark that the classification in \cite{DHS} deals with both {\em collapsed} and {\em non-collapsed} solutions.
Non-compact ancient (eternal) solutions to the 2-dimensional Ricci flow were classified by Daskalopoulos and Sesum in \cite{DS1} (see also in \cite{Ham3, Chu}).
It turns out that in this case the only eternal solutions with bounded curvature $R(\cdot,t)$ at each time slice, are the cigar solitons.  The results mentioned above classify all complete 2D-Ricci flow ancient solutions, both compact and non-compact.

\sk
Solutions analogous to the 2-dim Ricci flow King solution exist in the \emph{higher dimensional ($n \ge 3$) Yamabe flow} as well.  Again they are not solitons, although they are given in an explicit form
and were discovered by King \cite{K1} (and later independently by Brendle in a private communication with the authors).  However, in the case of the Yamabe flow many more solutions exist.  In \cite{DDKS1, DDKS2},
Daskalopoulos, del~Pino, King and Sesum constructed a five parameter family of type~I ancient solutions to the Yamabe flow on $S^n$, not given in closed form, however looking similar to the King solutions in \cite{K1}.
In fact the King solutions are part of this class of solutions that can be visualized as two shrinking solitons, coming from spatial infinities and glued together.
In addition in \cite{DDS}, Daskalopoulos, del~Pino, and Sesum constructed type~II compact ancient solutions to the Yamabe flow on $S^n$ (which they called a {\it tower of bubbles}).  Unlike the above mentioned type~I ancient solutions,
the Ricci curvature of the tower of bubbles solutions changes its sign (they still have nonnegative scalar curvature).
The above examples show that the classification of closed ancient solutions to the Yamabe flow is very difficult, if not impossible.

\sk

For \emph{Curve Shortening} (MCF for curves in the plane) Daskalopoulos, Hamilton, and Sesum \cite{DHS1} classified all ancient compact convex solutions by showing that the only possibilities are the shrinking circle and the paperclip solution.

The existence of higher dimensional ancient compact convex solutions was settled by White \cite{Wh} and then also by Haslhofer and Hershkovits \cite{HO} who constructed solutions that are not solitons, and for which no closed form seems to exist.  In \cite{AngFormal} formal asymptotics were derived for White and Haslhofer-Hershkovits' ancient solutions.  In \cite{ADS1} we showed that any ancient, rotationally symmetric, non-collapsed solution of MCF satisfies the {\em unique asymptotics as $t\to-\infty$.  }
The classification of such solutions was established in \cite{ADS},
where we show that every uniformly 2-convex ancient oval must be unique up to rotation, scaling, and translation and hence it must be the ancient oval solution constructed by White and by Haslhofer--Hershkovits (\cite{Wh,HO}).
This implies that every closed, uniformly 2-convex and non-collapsed ancient solution to the mean curvature flow must be either the family of contracting spheres or the unique, up to isometries, ancient oval.

The notion of {\em non-collapsedness} for \emph{mean convex Mean Curvature Flow} was introduced by X.J.~Wang (\cite{XW}) and subsequently was shown by Andrews (\cite{An}) to be preserved along MCF.
It is known (\cite{HK}) that all non-collapsed ancient compact solutions to MCF are convex.
The non-collapsedness condition turns out to be important in the classification of ancient compact convex solutions, as evidenced by the ``pancake type'' examples which become collapsed as $t\to -\infty$ (see \cite{BLT} and \cite{XW} for the existence of such solutions and \cite{BLT} for
a beautiful work on their classification under rotational symmetry).
On the other hand, ancient non compact non-collapsed uniformly 2-convex solutions were considered by Brendle and Choi in \cite{BC} and \cite{BC1}, where the authors show that any noncollapsed, uniformly 2-convex non compact ancient solution to the mean curvature flow must be a rotationally symmetric translating soliton, and hence the Bowl soliton, up to scaling and isometries.

\smallskip

In this paper we focus our attention on {\em 3-dimensional Ricci flow.  }
Consider an ancient compact solution to the 3-dimensional Ricci flow
\begin{equation}
  \label{eq-rf}
  \frac{\partial}{\partial t}g_{ij} = -2R_{ij},
\end{equation}
existing for $t\in (-\infty,T)$ so that it shrinks to a round point at $T$.

Perelman \cite{Pe1} established the existence of a rotationally symmetric ancient $\kappa$-noncollapsed solution on $S^3$.  This ancient solution is of type~II backward in time, in the sense that its scalar curvature satisfies
\[
\sup_{M\times (-\infty,0)} |t| \,|R(x,t)| = \infty
\]
In forward time the ancient solution forms a type~I singularity, as it shrinks to a round point.  Perelman's ancient solution has backward in time limits which are the Bryant soliton and the round cylinder $S^2\times \mathbb{R}$, depending on how the sequence of points and times about which one rescales are chosen.  These are the only backward in time limits of Perelman's ancient solution.
Let us remark that in contrast to the {\em collapsed } King ancient solution of the 2-dimensional Ricci flow, the Perelman ancient solution is
{\em non-collapsed}.  In fact there exist other ancient compact solutions to the 3-dimensional Ricci flow which are collapsed and the analogue of the King solution (see in \cite{Fa2}, \cite{BKN}).  Let us recall the notion of {\em $\kappa$-noncollapsed} metrics,
which was introduced by Perelman in \cite{Pe1}.

\begin{definition}[$\kappa$-noncollapsed property \cite{Pe1}]
  The metric $g$ is called {\em $\kappa$-noncollapsed} on the scale $\rho$, if every metric ball $B_r$ of radius $ r < \rho$ which satisfies
  $|\Rm | \leq r^{-2}$ on $B_r$ has volume at least $\kappa \, r^n$.
  For any \(\kappa>0\) an ancient Ricci flow solution is called $\kappa$-noncollapsed, if it is $\kappa$-noncollapsed on all scales $\rho$.
\end{definition}

It turns out that this is an important notion in the context of ancient solutions and singularities.  In fact, in \cite{Pe1} Perelman proved that every ancient solution arising as a blow-up limit at a singularity of the Ricci flow on compact manifolds is $\kappa$-noncollapsed on all scales for some
$\kappa >0$.  In this context, the following conjecture plays an important role in the classification of singularities.

\begin{conjecture}[Perelman]
  \label{conj-perelman}
  Let $(S^3,g(t))$ be a compact ancient $\kappa$-noncollapsed solution to the Ricci flow \eqref{eq-rf} on $S^3$.  Then $g(t)$ is either a family of contracting spheres or Perelman's solution.
\end{conjecture}

The Hamilton-Ivey pinching estimate tells us that any two or three dimensional ancient solution with bounded curvature at each time slice has nonnegative sectional curvature.  Since our solution $(S^3,g(t))$ is closed, the strong maximum principle implies
that the sectional curvatures, and hence the curvature operator, are all strictly positive.  Hence, Hamilton's Harnack estimate (\cite{Ha1}) implies that the scalar curvature of the solution $g(t)$ has $R_t \ge 0$, yielding that there exists a uniform constant $C > 0$ so that $R(\cdot,t) \le C$, for all $t\in (-\infty,t_0]$.  Since the curvature is positive, this yields the curvature bound
\begin{equation}
  \label{eq-curv-bound}
  \|\Rm\|_{g(t)} \le C, \qquad \mbox{for all} \,\,\, -\infty < t \le t_0,
\end{equation}
and for some uniform constant $C$.

The previous discussion implies that any closed 3-dimensional $\kappa$-noncollapsed ancient solution is actually a {\em $\kappa$-solution} in the sense that was defined by Perelman in \cite{Pe1} (see Section \ref{sec-preliminaries} for more details).

\sk
In a recent important paper \cite{Br}, Brendle proved that any three-dimensional {\em non compact } ancient $\kappa$-solution is isometric to either a family of shrinking cylinders or their quotients, or to the Bryant soliton.  Brendle first shows that all three-dimensional ancient $\kappa$-solutions that are non compact must be rotationally symmetric.
He then shows that such a rotationally symmetric solution, if it is not a cylinder or its quotient, must be a steady Ricci soliton.
In the same paper Brendle states that his techniques can easily be applied to obtain the rotational symmetry of
three-dimensional compact ancient $\kappa$-solutions.  In fact this is shown in detail in a very recent work by Brendle \cite{Br1}.
Also very recently Bamler and Kleiner in \cite{BK} obtained the same result in the compact case using different methods from Brendle.

\sk

The {\em summary} is that any solution satisfying the assumptions
of Conjecture \ref{conj-perelman} is rotationally symmetric, hence reducing the resolution of Conjecture \ref{conj-perelman} to establishing the classification of {\em rotationally symmetric} solutions.  The challenge in this problem comes from the fact that Perelman's solution is not given in explicit form and is not a soliton.  Similar challenge appears in the classification of ancient compact
non-collapsed MCF solutions which was resolved in \cite{ADS1}.
Because of the instability of the linearized operator at the cylinder,
which appears as asymptotic limit as $t \to -\infty$ in both flows, one needs to establish refined asymptotics for any ancient
solution under consideration as $t \to -\infty$.

\sk

\sk

In an attempt to resolve Conjecture \ref{conj-perelman}, we will establish in this paper the (unique up to scaling) asymptotic behavior of any reflection and rotationally symmetric compact $\kappa$-noncollapsed ancient solution to the Ricci flow on $S^3$ which is not isometric to a round sphere.  Our main result states as follows.

\begin{theorem}\label{thm-asym}
  Let $(S^3, g(t))$ be any reflection and rotationally symmetric compact, $\kappa$-noncollapsed ancient solution to the Ricci flow on $S^3$ which is not isometric to a round sphere.
  Then the rescaled profile $u(\sigma,\tau)$ (solution to equation \eqref{eq-u}), defined on $\mathbb{R}\times \mathbb{R}$, has the following asymptotic expansions:
  \begin{enumerate}
    \item[(i)] For every $L > 0$,
    \[
    u(\sigma,\tau) = \sqrt{2} \Bigl(1 - \frac{\sigma^2 - 2}{8|\tau|}\Bigr) + o(|\tau|^{-1}), \qquad {\mbox on} \,\, |\sigma| \le L
    \]
    as $\tau \to -\infty$.
    \sk
    \item[(ii)] Define $z := {\sigma}/{\sqrt{|\tau|}}$ and $\bar{u}(\sigma,\tau) := u(z\sqrt{|\tau|}, \tau)$.
    Then,
    \[
    \lim_{\tau \to -\infty} \bar{u}(z,\tau) = \sqrt{2 - \frac{z^2}2}
    \]
    uniformly on compact subsets of $|z| < 2$.
    \sk
    
    \item[(iii)] Let $k(t) := R(p_t,t)$ be the maximal scalar curvature which is attained at the two tips $p_t$, for $t \ll -1$.  Then the rescaled Ricci flow solutions $(S^3,\bar g_{t}(s), p_{t})$, with ${\ds \bar {g}_{t}(\cdot, s) = k(t) \, g(\cdot,t+k(t)^{-1}\, s)}$, converge to the unique Bryant translating soliton with maximal scalar curvature one.
    Furthermore, $k(t)$ and the diameter $d(t)$ satisfy the asymptotics
    \[
    \qquad k(t) = \frac{\log|t|}{|t|} (1 + o(1))\, \quad \mbox{and} \quad d(t) = 4\sqrt{|t|\log |t|}\, (1 +o(1))
    \]
    as $t \to -\infty.$
  \end{enumerate}
\end{theorem}\noindent

In a forthcoming work, we use Theorem \ref{thm-asym} to address Conjecture \ref{conj-perelman}, in a similar way that results about unique asymptotics of ancient ovals shown in \cite{ADS1} were used to prove the classification result of closed ancient mean curvature flow solutions (see \cite{ADS}).

\sk

As an immediate Corollary of the result in \cite{Br1} and of Theorem \ref{thm-asym} we have the following result.

\begin{corollary}
  Let $(S^3, g(t))$ be any reflection symmetric compact $\kappa$-noncollapsed ancient solution to the Ricci flow on $S^3$.  Then, it is rotationally symmetric and is either isometric to a round sphere or it has unique asymptotics which are given by Theorem \ref{thm-asym}.
\end{corollary}

\sk
In order to prove Theorem \ref{thm-asym} we will combine techniques developed in \cite{ADS1} and \cite{Br}.  In \cite{Br}, under the assumption on rotational symmetry, Brendle constructed barriers by using gradient Ricci solitons with singularity at the tip
which were found by Bryant (\cite{Bry}).  In the proof of Theorem \ref{thm-asym} we use Brendle's barriers to localize our equation
in the parabolic region, similarly as in \cite{Br}.  The methods are similar to the ones used in \cite{ADS1}, but new difficulties arise due to the fact
that the equation \eqref{eqn-un} is non-local.  The localization enables us to do spectral decomposition in the parabolic region and obtain refined asymptotics of our solution.

\sk

The {\em outline} of the paper is as follows: In section \ref{sec-preliminaries} we discuss the backward in time limit of our solution and we list all equations under rotational symmetry, introducing different change of variables in different regions.  In section \ref{sec-cylindrical} we use Brendle's barriers to achieve the spectral decomposition of our solution which yields precise asymptotics in the parabolic region.
Subsequently, we combine this exact behavior in the parabolic region together with barrier type arguments to obtain the precise behavior of our solution in the
intermediate region (see section \ref{section-intermediate}).  In the last section
\ref{sec-tip}, we show the convergence of our solution to the Bryant soliton, after
appropriate rescaling and change of variables, and obtain the precise behavior of the
maximum scalar curvature and the diameter, as time approaches $-\infty$.

\section{Preliminaries}

\label{sec-preliminaries}

\subsection{Equations under rotational symmetry}

\label{sec-equations}
Assume that $g$ is a solution of the Ricci flow \eqref{eq-rf} on $S^{3}$, which is rotationally and reflection symmetric and shrinks to a round point at time $T$.  Since $g(\cdot,t)$ is rotationally symmetric, it can can be written as
\[
g = \phi^2\, dx^2 + \psi^2\, g_{\rm can}, \qquad \mbox{on} \,\, (-1,1)\times S^2
\]
where $(-1,1)\times S^n$ may be naturally identified with the sphere $S^{3}$ with its North and South Poles removed.  The quantity $\psi(x,t) > 0$ is the radius of the hypersurface $\{x\}\times S^2$ at time $t$.  By the reflection symmetry assumption we have $\psi(x,t) =\psi(-x,t)$ for all $x\in (-1,1)$.
The distance function to the equator is given by
\[
s(x, t) = \int_0^x \phi(x', t)\, dx'.
\]
We will write
\[
s_\pm(t) := \lim_{x\to \pm 1} s(x, t),
\]
or shortly $s_{\pm}$, for the distance from the equator to the South and the North Poles, respectively.  Under Ricci flow the distances $s_\pm(t)$ evolve with time.
If we abbreviate
\[
ds = \phi(x, t) \, dx \qquad \text{ and }
\qquad \frac{\pd}{\pd s} = \frac 1{\phi(x, t)} \, \frac{\pd}{\pd x}
\]
then we can write our metric as
\begin{equation}\label{eq-metric}
  g = ds^2 + \psi^2\, g_{\rm can}.
\end{equation}
The time derivative does not commute with the $s$-derivative, and in general we
must use
\[
\frac{\pd}{\pd t}ds = \phi_t \, dx = \frac{\phi_t}{\phi} \, ds \qquad
\text{ and } \qquad
\left[ \frac{\pd}{\pd t}, \frac{\pd}{\pd s}\right] =
-\frac{\phi_t }{\phi} \frac{\pd}{\pd s}.
\]
The Ricci tensor is given by
\begin{align*}
  {\mathrm{Rc}}
  & = 2K_0 \, ds^2 + \left[K_0 + K_1\right] \psi^2 g_{\rm can} \\
  & = -2\frac{\psi_{ss}}{\psi }\, ds^2
  + \left\{
  -\psi \psi_{ss} - \psi_s ^2+1
  \right\} g_{\rm can}
\end{align*}
where $K_0$ and $K_1$ are the two distinguished sectional curvatures that any
metric of the form \eqref{eq-metric} has.  They are the curvature of a plane tangent to $\{s\}\times S^n$, given by
\begin{equation}
  \label{eq-K1}
  K_1:= \frac{1-\psi_s^2}{\psi^2},
\end{equation}
and the curvature of an orthogonal plane given by
\begin{equation}
  \label{eq-K0}
  K_0 := -\frac{\psi_{ss}}{\psi}.
\end{equation}
Moreover, the scalar curvature is given by
\[
R = g^{jk} R_{jk}
= 4K_0 + 2K_1.
\]
The time derivative of the metric is
\[
\frac{\pd g}{\pd t} =
2\frac{\phi_t}{\phi}\, ds^2 + 2\psi\psi_t\, g_{\rm can}.
\]
Therefore, since the metric $g(t)$ evolves by Ricci flow \eqref{eq-rf}, we have
\[
\frac{\phi_t}{\phi} = 2\frac{\psi_{ss}}{\psi},
\]
so that
\[
\pdd t \, ds = 2\frac{\psi_{ss}}{\psi}\, ds \qquad \text{ and } \qquad
\left[\pdd t, \pdd s\right]
= -2\frac{\psi_{ss}}{\psi}\, \pdd s.
\]
Under Ricci flow the radius $\psi(s,t)$ satisfies the equation
\begin{equation}
  \label{eq-psi}
  \psi_t = \psi_{ss} - \frac{1-\psi_s^2}{\psi}.
\end{equation}
As in \cite{AK1}, for our metric \eqref{eq-metric} to define a smooth metric on $S^3$ we need to have
\begin{equation}
  \label{closing-psi}
  \psi_s(s_-) = 1, \,\,\,\, \psi^{(2k)}(s_-) = 0 \qquad \mbox{and} \qquad \psi_s(s_+) = -1, \,\,\,\, \psi^{(2k)}(s_+) = 0,
\end{equation}
for $k\in \mathbb{N}\cup\{0\}$.

\sk

Consider next the rescaled function
\begin{equation}\label{eqn-defnu}
  u(x, \tau) := \frac{\psi(x,t)}{\sqrt{T-t}},
\end{equation}
as well as the rescaled time and distance to the equator
\[
\sigma(x,t) := \frac{s(x,t)}{\sqrt{T-t}}, \quad \tau = -\log (T-t).
\]
If we write
\[
d\sigma = \frac{ds}{\sqrt{T-t}}, \qquad
\frac{\pd}{\pd\sigma} = \sqrt{T-t}\,\frac{\pd}{\pd s},\qquad
\pdd\tau = (T-t)\pdd t,
\]
then we get
\begin{equation} \label{eq-dsigma}
  \pdd\tau d\sigma = \left(\frac12 + 2\frac{u_{\sigma\sigma}}{u}\right)d\sigma.
\end{equation}
For the commutator we get
\begin{equation}\label{eq-sigma-tau-commutator}
  \left[\pdd\tau, \pdd\sigma\right]
  = - \left(\frac12 + 2\frac{u_{\sigma\sigma}}{u}\right) \pdd\sigma.
\end{equation}
We write $\sigma_{\pm}(t) = \frac{s_{\pm}(t)}{\sqrt{T-t}}$, or simply $\sigma_{\pm}$ for the rescaled distance from the equator to the poles.
The rescaled radius $u:(-1,1)\times(-\infty, 0)\to\R$ satisfies the equation
\begin{equation} \label{eq-u}
  u_{\tau} = u_{\sigma\sigma} + \frac{u_{\sigma}^2}{u} - \frac 1u + \frac u2
\end{equation}
with boundary conditions
\begin{equation} \label{eq-closing1}
  u_{\sigma} = \mp 1,\quad
  u^{(2k)}=0 \qquad \text{ at }\sigma=\sigma_\pm(\tau), \quad \tau<0.
\end{equation}

\subsubsection*{The expansion term}
By analogy with similar equations in MCF one might expect a term of the form $-\frac 12 \sigma u_\sigma$ in \eqref{eq-u}.  However, in the present set-up that term is replaced by a change in the commutator $[\pd_\tau, \pd_\sigma]$ (in \eqref{eq-sigma-tau-commutator} or \eqref{eq-dsigma}) which accounts for the stretching involved in passing from the $s$ to $\sigma$ coordinate.

\sk

The vector fields $\partial_{\tau}$ and $\partial_{\sigma}$ do not commute.  To overcome this we replace the vector field $\pd_\tau$ with
\begin{equation}
  \label{eq-comm-vec-field}
  D_\tau = \pd_\tau - I \, \pd_\sigma
\end{equation}
for some function $I$, and we require that $D_\tau$ and $\pd_\sigma$ commute.  It follows from
\[
[D_\tau, \pd_\sigma]
= [\pd_\tau, \pd_\sigma] + I_\sigma \pd_\sigma
= \left\{-\frac12 - 2\frac{u_{\sigma\sigma}}{u} + I_\sigma\right\} \pd_\sigma.
\]
Hence $D_\tau$ and $\pd_\sigma$ commute if
\begin{equation}
  \label{eqn-defn-J}
  I(\sigma, t) = \frac12\sigma + J(\sigma, t),\qquad
  J(\sigma,t) := 2 \int_{0}^{\sigma} \frac{u_{\sigma\sigma}}{u}\, d\sigma'.
\end{equation}
\begin{figure}
  \includegraphics[width=\textwidth]{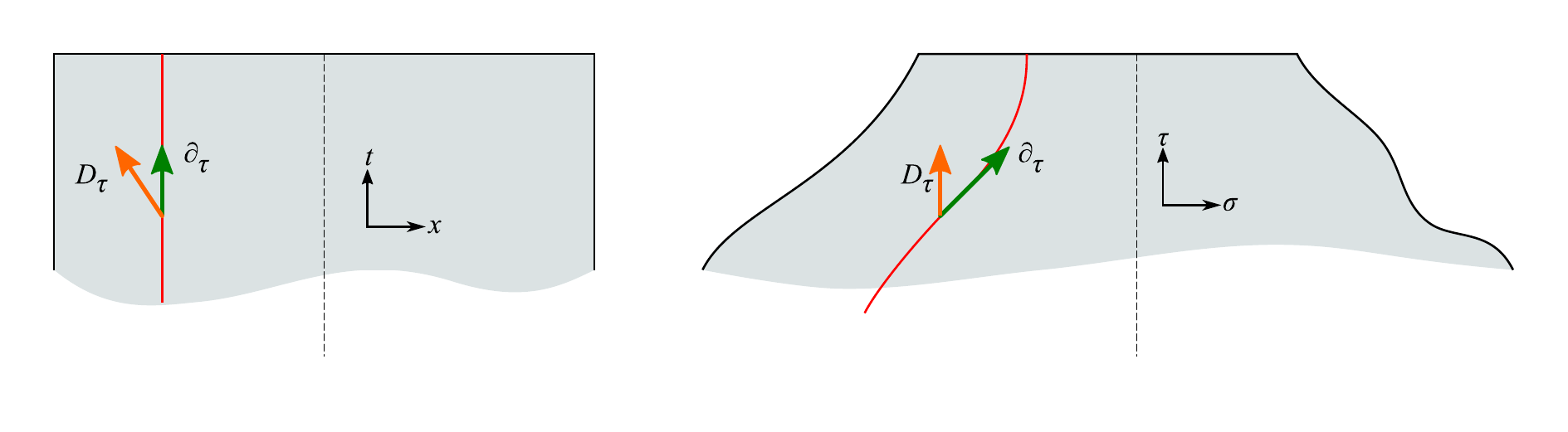}
  \caption{Space-time in $(x, \tau)$ coordinates on the left, and in $(\sigma,\tau)$ coordinates on the right.  $\pd_\tau$ and $\pd_x$ commute, and $D_\tau$ and $\pd_\sigma$ commute.}
\end{figure}%
We can then write the equation \eqref{eq-u} for $u$ as
\begin{equation}\label{eqn-u0}
  D_{\tau} u = u_{\sigma\sigma} - \frac{\sigma}{2} \, u_{\sigma}
  - J(\sigma,\tau) \, u_{\sigma} + \frac{u_{\sigma}^2}{u} - \frac 1u + \frac u2.
\end{equation}
One can think of $D_\tau$ as ``the derivative with respect to $\tau$ keeping $\sigma$ constant,'' while $\pd_\tau$ is ``the $\tau$-derivative keeping $x$ fixed.''
We will abuse notation and write $u_\tau$ both for $\pd_\tau u$ and for $D_\tau u$, when it is clear from the context which time derivative we mean.
For instance, we will write equation \eqref{eqn-u0} as
\begin{equation}\label{eqn-un}
  u_\tau = u_{\sigma\sigma} - \frac{\sigma}{2} \, u_{\sigma}
  - J(\sigma,\tau) \, u_{\sigma} + \frac{u_{\sigma}^2}{u} - \frac 1u + \frac u2.
\end{equation}

\begin{remark}\label{rem-cylinder}
  The solution $u \equiv \sqrt{2}$ corresponds to the shrinking cylinder soliton $S^2\times \mathbb{R}$.
\end{remark}

The representation $g=\phi^2 \, dx^2+\psi^2\, g_{\rm can}$ leads to singularities at the poles $x=\pm 1$.  We overcome this difficulty by choosing new local coordinates.  As in Angenent-Caputo-Knopf \cite{ACK}, we regard $\psi(s,t)$ as a new local coordinate on any interval $(x_0, x_1)$ on which $\psi(x, t)$ is a monotone function of $x$.
By our assumptions $\psi(x,t)$ is a concave function of $s$ (i.e.~$\psi_{ss}<0$).  This implies $\psi(x, t)$ is strictly increasing for $x<0$ and strictly decreasing for $x>0$.  In the region $x<0$ (the southern hemisphere) we take $\psi$ as coordinate, and express the metric and all its components as functions of $(\psi,t)$.  The metric is then
\begin{equation}\label{eqn-ggg}
  g = \phi^2\, dx^2 + \psi^2\,g_{\rm can},
\end{equation}
where we still have to write $(\phi\, dx)^2$ in terms of $\psi$.  We have
\[
\phi(x, t) \, dx = ds = \frac{d\psi}{\psi_s}
\]
so that
\[
g = \frac{d\psi^2}{\psi_s^2} + \psi^2\, g_{\rm can}.
\]
To describe the evolution of the metric we must therefore keep track of the quantity $\bar Y := \psi_s^2$ as a function of $(\psi, t)$.  That is, we set
\begin{equation}
  \label{eqn-coor-Y}
  \bar Y(\psi, t) = \psi_s^2(s,t), \qquad \psi=\psi(s,t).
\end{equation}
A direct calculation shows that it evolves by the PDE
\begin{equation}\label{eq-bY}
  \bar Y_t =
  \bar Y \bar Y_{\psi\psi}
  - \frac 12\, (\bar Y_{\psi})^2
  + (1 - \bar Y)\, \frac{\bar Y_{\psi}}{\psi}
  + 2 (1-\bar Y) \, \frac{\bar Y}{\psi^2}.
\end{equation}
In this equation $\pd_t$ stands for the derivative with respect to time at constant $\psi$.  Therefore $\pd_t$ and $\pd_\psi$ commute.

As above, we will also work here with the rescaled variables $(u,\tau)$ given by \eqref{eqn-defnu}.  Thus we introduce $Y:=u_\sigma^2$
as a function of $(u,\tau)$, that is
\begin{equation}\label{eqn-coor-tip}
  \bar Y(\psi, t) = Y(u, \tau), \qquad u=\frac{\psi}{\sqrt{T-t}},\, \tau = -\log(T-t).
\end{equation}
A short computation then shows that $Y(u,\tau)$ evolves by
\begin{equation}\label{eqn-Y}
  Y_\tau + \frac u2 \, Y_u = Y Y_{uu} - \frac 12\, (Y_u)^2 + (1 - Y)\, \frac{Y_u}{u} + 2 (1-Y) \, \frac{Y}{u^2}.
\end{equation}

\sk

\subsection{Backward limit of our solution}\label{subsec-blimit}

We will first analyze the backward limit of any $\kappa$-solution.  Let $\bar{t} = -t$.  By the work of Perelman (\cite{Pe1}, see also \cite{KL} for details), for every $\bar{t} > 0$ there is some $q(\bar{t}) \in M$ so that $l(q(\bar{t}),\bar{t}) \le 3/2$, where
$l(q,\bar{t})$ denotes the reduced distance
\[
l(q,\bar{t}) := \frac{1}{2\sqrt{\bar{t}}} \inf_{\gamma} \int_0^{\bar{t}} \sqrt{\bar{t}'}\,(R(\gamma(\bar{t}')) + |\dot{\gamma}|^2)\, d\bar{t}',
\]
and where $R(\gamma(\bar{t}'), \bar{t}')$ and $|\dot{\gamma}|$ are computed using the metric at $-\bar{t}'$ and $\gamma$ is any curve connecting some fixed point $p$ and any point $q$ so that $\gamma(0) = p$ and $\gamma(\bar{t}) = q$.  In \cite{Pe1} Perelman showed that for a subsequence of $\bar{t}_i := -t_i$ (call it again $\bar{t}_i$), the parabolically rescaled sequence of metrics $g_i(t)$, around $l(q(\bar{t}_i),\bar{t}_i), t_i)$, converges to a non flat gradient shrinking Ricci soliton.  By the classification result of such solitons we know they are either the spheres or round cylinders $S^2\times \mathbb{R}$.  The limiting gradient soliton is called an {\em asymptotic soliton.  }

\begin{lemma}
  \label{lemma-sphere}
  Assume $(S^3,g(t))$ is a closed $\kappa$-solution whose asymptotic soliton is a round sphere.  Then the Ricci flow solution $g(\cdot,t), t \in
  (-\infty,T)$ must itself be a family of shrinking round spheres.
\end{lemma}

\begin{proof}
  By our assumption, $(S^3, g_i(t))$, where $g_i(\cdot,t) = \frac{1}{(-t_i)}\, g(t_i - t \, t_i)$, converges to a sphere as $i\to \infty$.  Then there exist a $\rho > 0$ and an $i_0 > 0$, so that for all $i\ge i_0$ we have $\ric(g_i(0)) \ge \rho\, R(g_i(0))\, g_i(0)$, or equivalently, $\ric(g(t_i)) \ge \rho R(g(t_i))\, g(t_i)$.  By Theorem 9.6 in \cite{Ha} we have that $\ric(g(t)) \ge \rho \, R(g(t))\, g(t)$ holds for all $t \ge t_i$.  Since $t_i\to -\infty$, we get that
  \[
  \ric(g(t)) \ge \rho R(g(t))\, g(t), \qquad \mbox{for all} \quad t \in (-\infty,T).
  \]
By the result in \cite{BHS}, this yields that $(S^3,g(t))$ has constant positive sectional curvature, for all $t \in (-\infty,T)$, and hence our solution is a family of shrinking spheres.  \end{proof}

\sk

Let us now assume that $(S^3, g)$ is a $\kappa$-solution which is rotationally and reflection symmetric.  We have seen in
the previous subsection that $g$ may be expressed in the form \eqref{eqn-ggg} and that because of reflection symmetry $x = 0$ can be taken to correspond to a point of {\em maximum radius} $\psi$ for every $t \le t_0$.  We claim the following.

\begin{lemma}
  \label{lemma-point-q}
  There exist uniform constants $C < \infty$ and $t_0 \ll -1$ so that
  \[
  R(q,t) \le \frac{C}{|t|}, \qquad \mbox{\em for all} \, \,\, t \leq t_0
  \]
  where $q\in M$ corresponds to $x = 0$ and hence $s = 0$, for all $t \le t_0$.
\end{lemma}

\begin{proof}
  By the maximum principle applied to $\psi$ that satisfies \eqref{eq-psi} we have
  \[
  \frac{d}{dt}\psi_{\max} \le - \frac{1}{\psi_{\max}}.
  \]
  Integrating this differential inequality on $[t, t_1]$, for a fixed $t_1$, yields that $\psi_{\max}^2(t) \ge 2|t| + \psi^2_{\max}(t_1) - 2 |\bar t_1| \geq |t|$,
  for all $t \le t_0 \ll -1$.  This implies the bound
  \begin{equation}
    \label{eq-K1-bound}
    K_1(q,t) = \frac{1 -\psi_s^2}{\psi^2} \le \frac{1}{|t|}, \qquad \mbox{for all}\,\, t \leq t_0 \ll -1.
  \end{equation}
  
  Next, we claim there exist uniform constants $C$ and $t_0 \ll -1$ so that
  \begin{equation}\label{eqn-K10}
    K_0(q,t) \le C\, K_1(q,t), \qquad \mbox{for all}\,\, t \leq t_0 \ll -1.
  \end{equation}
  To prove this claim we argue by
  contradiction.  Assume that the claim is not true, meaning there exist a sequence
  $t_i \to -\infty$ and $C_i \to \infty$ so that \begin{equation}
    \label{eq-ratio-big}
    K_0(q,t_i) \ge C_i K_1(q,t_i).
  \end{equation}
  We rescale the Ricci flow solution $g(\cdot,t)$ by $R(0,t_i)$ around $(q,t_i)$, where
  $q\in M$ corresponds to $s = 0$.  By Perelman's compactness theorem for
  $\kappa$-solutions, there exists a subsequence of rescaled solutions converging to
  another rotationally symmetric $\kappa$-solution, which, in view of
  \eqref{eq-ratio-big}, implies that on the limiting solution one has
  $K^{\infty}_1 \equiv 0$ and $K^\infty_0\not\equiv0$ on the limit
  $M_{\infty}$.  Since our limiting rotationally symmetric metric $g_{\infty}$ is of
  the form $g_{\infty} = ds_{\infty}^2 + \psi_{\infty}(s)^2 g_{\rm can}$, where
  $K^{\infty}_1 = \frac{1 - \psi_{\infty s}^2}{\psi_{\infty}^2} \equiv 0$, we get
  $\psi^2_{\infty s} \equiv 1$.  This implies that
  $\psi_{\infty,ss}\equiv 0$, which contradicts $K^\infty_0\not\equiv0$.
  We conclude that \eqref{eqn-K10} holds which combined with \eqref{eq-K1-bound} finishes the proof of the Lemma.
\end{proof}

By the previous discussion in this section (in particular Lemma \ref{lemma-sphere}), we may assume in the rest of the paper that {\em the asymptotic soliton of our $\kappa$-solution is a round cylinder}.
Define the parabolically rescaled metric
\[
\tilde{g}(\cdot,\tau) := \frac{1}{(-t)}\, g(\cdot,t), \qquad \tau = -\log(T - t).
\]
Let $q \in S^3$ be the point as in Lemma \ref{lemma-point-q}.  Then we have the following result.

\begin{proposition}[Proposition 3.1 in \cite{Br}]
  \label{prop-back-limit}
  Let $(S^3,g(\cdot,t))$ be a closed $\kappa$-solution whose asymptotic soliton is a round cylinder.  Then the rescaled solution
  $(S^3, \tilde{g}(\cdot,\tau))$ around a fixed point $q$ converges in Cheeger-Gromov sense to the round cylinder of radius $\sqrt{2}$.
\end{proposition}

\begin{proof}
  See Proposition 3.1 in \cite{Br}.
\end{proof}

\subsection{Radially symmetric barriers from \cite{Br}}

S.  Brendle in \cite{Br} constructed barriers for equation \eqref{eqn-Y}.  His construction gives the following result which will be used in the next sections.

\begin{prop}[Brendle, Propositions 2.4 and 2.5 in \cite{Br}]\label{prop-barriers}
  There exists a one parameter family $Y_a:=Y_a(u)$ of super-solutions to the elliptic equation
  \begin{equation}\label{eqn-Ya-Simon}
    Y_a Y_a'' - \frac u2 \, Y_a' - \frac 12\, (Y_a')^2 + (1 - Y)\, \frac{Y_a'}{u} + 2 (1-Y_a) \, \frac{Y_a}{u^2} <0
  \end{equation}
  which are defined on $u \in [r^* a^{-1}, \frac 98 \sqrt{2} ]$, for a fixed number $r^* >0$.  Moreover, there exists a
  small constant $\eta >0$ which is independent of $a$ and a smooth function $\zeta(s)$ with $2+\zeta(\sqrt{2}) =1/4$
  such that
  \begin{equation}\label{eqn-asym-Ya}
    Y_a(u) = a^{-2} \, (2 u^{-2} -1) + a^{-4} \, (8 u^{-4} + \zeta(u) ) + O(a^{-6})
  \end{equation}
  for all $u \in [\sqrt{2}-\eta, \sqrt{2}+\eta]$ and $a \gg 1$.
  
\end{prop}

Note that Brendle in \cite{Br} considers solutions of
\[
Y_a Y_a'' - \frac u2 \, Y_a' - \frac 12\, (Y_a')^2 + (1 - Y_a)\, \frac{Y_a'}{u} + 2 (1-Y_a) \, \frac{Y_a}{u^2} <0
\]
on $u \in [r^* a^{-1}, \frac 98 ]$ since he scales ${\ds u= \frac{\psi}{\sqrt{-2t}}}$, but we will consider solutions of \eqref{eqn-Y} on
$u \in [a^{-1}\, r^*, \frac 98 \sqrt{2} ]$, for a different constant $r^*$, to be consistent with our scaling ${\ds u= \frac{\psi}{\sqrt{-t}}}$.
We refer the reader to \cite{Br} Section 2 (Propositions 2.4 and 2.5) for the details of the construction of the barriers $Y_a$
and their properties.

\smallskip

We next state the following result, which is Proposition 2.8 in \cite{Br} adopted to our notation.

\begin{proposition}[Brendle, Proposition 2.8 in \cite{Br}]\label{prop-Br1}
  There exists a large number $K$ with the following property.  Suppose that $a \ge K$ and $ \bar \tau \ll -1$.  Moreover, suppose that $\bar u(\tau)$ is a function satisfying $\left| \bar u(\tau) - \sqrt{2} \right| \le \frac{1}{100}\, a^{-2}$ and $Y(\bar u(\tau),t) \le \frac{1}{32}\, a^{-4}$ for all $\tau \le \bar{\tau}$.  Then
  \[
  Y(u,\tau ) \le Y_a(u), \qquad \mbox{for all} \quad r_* a^{-1} \le u \le \bar u(\tau), \, \,\, \tau \le \bar{\tau}.
  \]
  In particular, $Y(u,\tau) \le C\, a^{-2}$ for all $\frac12 \le u \le \bar u(\tau)$ and $\tau \le \bar{\tau}$.
\end{proposition}

The above proposition will play crucial role in Section \ref{sec-cylindrical} as it will allow us to introduce cut-off functions supported in the parabolic region of our solution.  For this application we will
need to have the result of the Proposition \ref{prop-Br1} holding up to the maximum point of our solution $u(\sigma,\tau) $ of \eqref{eq-u}.  Since we have assumed reflection symmetry this corresponds to $\sigma=0$.
However at the maximum $u(0,\tau)$ we have $Y:=u_\sigma^2=0$ and equation \eqref{eqn-Y} becomes degenerate.  In the following consequence of Proposition \ref{prop-Br1}, we justify that the comparison principle
can be extended up to $Y=0$ and allow us to have the result in this proposition with $u(\tau):= u(0,\tau) = \max u(\cdot,\tau)$.

\smallskip

\begin{proposition}[Corollary of Proposition 2.8 in \cite{Br}]\label{prop-Br2}
  There exists a large number $K$ with the following property.  Suppose that $a \ge K$ and $ \bar \tau \ll -1$.
  Moreover, suppose that $\left| u(0, \tau) - \sqrt{2} \right| \le \frac{1}{200}\, a^{-2}$ for all $\tau \le \bar{\tau}$.  Then
  \[
  Y(u,\tau ) \le Y_a(u), \qquad \mbox{for all} \quad r_* a^{-1} \le u \le u(0,\tau), \,\,\, \tau \le \bar{\tau}.
  \]
  In particular, $Y(u,\tau) \le C\, a^{-2}$ for $\frac12 \le u \le u(0,\tau)$ and all $\tau \le \bar{\tau}$.
\end{proposition}

\begin{proof} We will apply Proposition \ref{prop-Br1} for $u(\sigma,\tau)$, with $|\sigma| \leq \delta \ll 1$, and then let $\sigma \to 0$.  Since we have assumed reflection symmetry it is sufficient to assume that $\sigma >0$.
  Let $a >0$ be such that $\left| u(0, \tau) - \sqrt{2} \right| \le \frac{1}{200}\, a^{-2}$ for all $\tau \le \bar{\tau}$.  The uniform continuity
  of both $u (\cdot, \tau)$ and $ u_\sigma(\cdot,\tau)$ on $|\sigma | \leq 1, \tau \leq \bar \tau$ (this follows from the smooth uniform convergence on compact sets
  $\lim_{\tau \to -\infty} u (\cdot, \tau) = \sqrt{2}$) combined with $u_\sigma(0,\tau)=0$ imply that there exists $\delta(\bar \tau, a)$ such that
  \[
  \left| u(\sigma, \tau) - \sqrt{2} \right| \le \frac{1}{100}\, a^{-2} \qquad \mbox{and} \qquad |u_\sigma^2(\sigma,\tau)| \leq \frac 1{32} a^{-4}
  \]
  for all $|\sigma| \leq \delta(\bar \tau, a)$, $\tau \leq \bar \tau$.  It follows that the assumptions of Proposition \ref{prop-Br1}
  are satisfied for $u(\sigma,\tau)$ (recall that $Y(u(\sigma, \tau),\tau) = u_\sigma^2(\sigma,\tau)$).  Therefore, Proposition \ref{prop-Br1}
  yields that for all $0 < \sigma < \delta(\bar \tau, a)$,
  \[
  Y(u,\tau ) \le Y_a(u), \qquad \mbox{for all} \quad r_* a^{-1} \le u \le u(\sigma,\tau), \,\,\, \tau \le \bar{\tau}.
  \]
  Letting $\sigma \to 0$ we conclude the desired result.
\end{proof}

\smallskip

\section{Parabolic region asymptotics}

\label{sec-cylindrical}

Our goal in the next two sections is to establish the asymptotic behavior of any
rotationally symmetric solution $u(\sigma,\tau)$ of \eqref{eq-u} (or
equivalently \eqref{eqn-un}) in the cylindrical region
\[
\cC_\theta := \bigl\{ \, \sigma : | \,\, u(\sigma, \tau) \geq \theta/2 \, \bigr\}
\]
for any $\theta >0$ small.  This will be done in two steps: in this section, we will first
analyze the linearized operator or our equation at the cylinder $u=\sqrt{2}$ to
establish precise asymptotics of $u(\sigma,\tau)$ in the parabolic region
$\cP_L := \big \{ \, \sigma : \,\, |\sigma | \leq L \, \big \}$, holding for any
$L \gg 1$ and then, in the next section, we will use these asymptotics and barrier arguments to
establish the behavior of $u(\sigma,\tau)$ in the intermediate region
$\cI_{L,\theta} := \big \{ \, \sigma :\,\, |\sigma| \geq L, \,\, u(\sigma,\tau)
\geq \ \theta/2 \, \big \}.$ Because of reflection symmetry it is enough to consider only the case where $\sigma \ge 0$.

\sk

We will assume throughout this section that $u(\sigma,\tau) $ is a solution of
\eqref{eqn-un} (in commuting variables).  The cylindrical norm in the parabolic region is defined as
\begin{equation}
  \label{eq-norm-parabolic}
  \sup_{\tau' \le \tau} \|f(\cdot,\tau')\|,
\end{equation}
where $f: \R \times (-\infty,\tau_0] \to \R$ and
\[
\|f(\cdot,s)\|^2 
:= \int_\R f(\sigma,s)^2 \, e^{-\frac{\sigma^2}{4}} \, d\sigma.
\]
We will often denote $d\mu := e^{-\frac{\sigma^2}{4}}\, d\sigma$.  We will prove the following:

\begin{theorem}\label{prop-parabolic-asym} 
  Under the same assumptions as in Theorem \ref{thm-asym}, for any $L \gg 1$, the solution $u(\sigma,\tau)$ of \eqref{eqn-un} satisfies the following asymptotics
  \begin{equation}
    \label{eqn-u-asym-p}
    u(\sigma,\tau) 
    = \sqrt{2} \, \Bigl( 1 - \frac{\sigma^2-2}{8 |\tau|} 
    + o\bigl (\frac 1{|\tau|} \bigr) \Bigr)
    \qquad (\tau\to-\infty)
  \end{equation}
  uniformly for $|\sigma| \leq L$.
\end{theorem}

To prove Theorem \ref{prop-parabolic-asym} we will use crucial ideas from a
recent paper by Brendle \cite{Br} and combine them with our methods in \cite{ADS1}.  In
what follows below we will outline those crucial ideas from \cite{Br}, and for
the proofs of those we refer the reader to the same paper.

\sk

To prove the asymptotics
estimates in \cite{ADS1} in the compact case and in \cite{BC} and \cite{BC1} in the
non-compact case, the spectral decomposition of the linearized operator at the cylinder in terms of
Hermite polynomials has been used.  The localization argument used in both papers, \cite{ADS1} and \cite{Br}, to make the spectral decomposition possible uses a calibration argument which has no obvious analogue in the Ricci flow.  The main obstruction comes from the non-local character of \eqref{eqn-un},
as the construction of barriers plays an important role.
In \cite{Br}, Brendle manages

to use barrier arguments.  For the barrier construction he uses steady gradient Ricci
solitons with singularity at the tip, which were found by Bryant in \cite{Bry}.
The same technique plays an important role in our case.

In order to set up the barrier argument we will change the variables and
consider our evolving metrics written in the form
$\tilde{g} = Y(u,\tau)^{-1} (du)^2 + u^2 g_{S^2}$, where
$Y(u,\tau) = u_{\sigma}^2$.  Note also that $u_{\sigma} = \psi_s$, so will
sometimes consider $Y$ as a function of $\psi$.  Since $u_{\sigma} = 0$ at
$\sigma = 0$, this change of variables is good from $u = 0$ all the way up to
$u = u(0,\tau)$, but not including $u = u(0,\tau)$.

\subsection{Analysis near the cylinder soliton}
Let $v(\sigma,\tau)$ be such that
\begin{equation}
  \label{eq-v-defined}
  u(\sigma,\tau) = \sqrt{2} \, \bigl(1 + v(\sigma,\tau)\bigr).
\end{equation}
The function $v$ satisfies
\begin{equation}
  \label{eq-v-evolution}
  v_\tau = v_{\sigma\sigma}
  -\Bigl(\frac{\sigma}{2}+J[v]\Bigr) v_\sigma
  +v
  +\frac{v_\sigma^2}{1+v}
  - \frac{v^2}{2(1+v)},
\end{equation}
where the function $J[v]$ is defined by
\begin{equation}
  \label{eq-J}
  J[v](x, t) = 2\int_0^x \frac{v_{\sigma\sigma}}{1+v}\, d\sigma
  =2\int_0^x \frac{dv_\sigma}{1+v}.
\end{equation}
We can split the terms in \eqref{eq-v-evolution} into linear and higher order terms as follows
\begin{equation}
  \label{eq-v-evol-lin-plus-nonlin}
  v_\tau = \cL [v] + \cN [v]
\end{equation}
where
\begin{equation}\label{eqn-oper}
  \cL [v] := v_{\sigma\sigma} -\frac{\sigma}{2} v_\sigma + v
\end{equation}
and
\[
\cN [v] := -J[v]\, v_\sigma +\frac{v_\sigma^2}{1+v} - \frac{v^2}{2(1+v)}.
\]

\sk

The function $v(\sigma,\tau)$ is not defined for all $\sigma\in\R$ and therefore does not belong to the Hilbert space
$ \cH= L^2(\R , e^{-\sigma^2/4}\, d\sigma)$ at any time $\tau$.  Since this is the natural Hilbert space to consider, we truncate $v$ outside an interval $|\sigma|\geq \sigma_*(\tau)$ for a suitably chosen $\sigma_*(\tau)$, which we will allow to depend on the size of $v(\sigma,\tau)$ at some fixed finite value of $\sigma$.
Thus we follow Brendle \cite{Br} and define for each $\tau$
\begin{align}
  \label{eq-delta-tau}
  \delta(\tau)
  &:= \sup_{\tau'\leq\tau} \Big (
  \big|u(0, \tau')-\sqrt{2}\big| + |u_\sigma(0,\tau')|
  \Big ) = \sqrt{2}\, \sup_{\tau'\le\tau}
  |v(0, \tau')|
\end{align}
since $u_{\sigma}(0,\tau) = 0$ for all $\tau$.
\sk

By Lemma 3.8 in \cite{Br} we have that there exists a uniform constant $C$ so that for $\tau \le \tau_0 \ll -1$ we have
\sk
\begin{equation}
  \label{simon-error}
  \|\chi_{[-\delta^{-\theta}, \delta^{-\theta}]} (v_{\tau} - \mathcal{L} v)\|^2 \le C\, \delta^{\theta}\, \|\chi_{[-\frac12{\delta^{-\theta}}, \frac12{\delta^{-\theta}}]} v\|^2 + C\, e^{-\frac18\, \delta^{-2\theta}},
\end{equation}
\sk
where $\chi_{[-\delta^{-\theta}, \delta^{-\theta}]}$ and $\chi_{[-\frac12{\delta^{-\theta}}, \frac12{\delta^{-\theta}}]}$ are the characteristic functions of the sets $[-\delta^{-\theta}, \delta^{-\theta}]$ and $[-\frac12{\delta^{-\theta}}, \frac12{\delta^{-\theta}}]$, respectively, and $\theta = \frac{1}{100}$.  Note that we have suppressed the dependence of $\delta$ on $\tau$ and we simply wrote $\delta$ for $\delta(\tau)$.  We will do that below as well when there is no ambiguity.

Choose a cutoff function $\hat\chi\in C^\infty(\R)$ satisfying $\hat\chi = 1$ on $[-\frac12,\frac12]$ and $\hat\chi = 0$ outside of $[-1,1]$, and set
\begin{equation}
  \label{eq-chi-def}
  \chi(\sigma,\tau) = \hat\chi(\delta(\tau)^\theta\sigma),\qquad
  \theta=\frac{1}{100}.
\end{equation}
We can now introduce our truncated version of $v$:
\begin{equation} \label{eq-vbar-defined}
  \bar{v}(\sigma,\tau) := v(\sigma,\tau) \chi(\sigma, \tau).
\end{equation}
It easily follows that
\begin{equation}
  \label{eq-bar-v-another}
  \bar{v}_{\tau} - \mathcal{L} \bar{v} = \chi (v_{\tau} - \mathcal{L} v) + (\chi_{\tau} - \chi_{\sigma\sigma} + \frac{\sigma}{2} \chi_{\sigma})\, v - 2\chi_{\sigma} v_{\sigma} =: E_0(\sigma,\tau).
\end{equation}

We claim that there exists a uniform constant $C$ so that for all $\tau \le \tau_0 \ll -1$ we have
\begin{equation}
  \label{eq-E0-simon}
  \|E_0(\cdot,\tau)\|^2 \le C\, \delta^{\theta} \|\bar{v}\|^2 + C\, e^{-\frac18\, \delta^{-2\theta}}.
\end{equation}
Indeed, by \eqref{simon-error} we immediately get that
\[
\|\chi (v_{\tau} - \mathcal{L} v)\|^2 \le C\, \delta^{\theta}\, \|\bar{v}\|^2 + C \, e^{-\frac18\, \delta^{-2\theta}}.
\]
By the definition of $\chi$ in \eqref{eq-chi-def} we have
\[
\|(\chi_{\tau} - \chi_{\sigma\sigma} + \frac{\sigma}{2}\, \chi_{\sigma})\, v - 2\chi_{\sigma} v_{\sigma}\|^2 \le C\, e^{-\frac18\, \delta^{-2\theta}}.
\]
Combining those two estimates immediately yields \eqref{eq-E0-simon}.

\sk


\begin{lemma}\label{prop-delta-Lipschitz}
  The function $\delta(\tau)$ is non-decreasing with $\lim_{\tau\to -\infty} \delta(\tau) = 0$.  Moreover, $\delta(\tau)$ is Lipschitz continuous; in particular, $\delta(\tau)$ is absolutely continuous and its derivative $\delta'(\tau)$ is uniformly bounded and non negative a.e.
\end{lemma}
\begin{proof}
  The definition of $\delta(\tau)$ directly implies that $\delta(\tau)$ is non-decreasing.
  As $\tau\to-\infty$ we have $u(\sigma,\tau)\to\sqrt{2}$ uniformly for $|\sigma|\leq 2\sigma_0$.
  By parabolic regularity it follows that all derivatives of $u$ are uniformly bounded for $|\sigma|\leq \sigma_0$ and for $\tau\ll 0$.
  Hence $u(0, \tau)$ is a smooth function of $\tau$ with uniformly bounded time derivatives.
  This implies that $\delta(\tau)$ is a Lipschitz function of $\tau$.
  Rademacher's Theorem implies that $\delta(\tau)$ is absolutely continuous and, in particular, differentiable almost everywhere.
\end{proof}

We have the following crucial lemma which will allow us to control error terms coming from cutoff functions.
\begin{lemma}
  \label{lemma-simon-crucial} There exist uniform constants $\tau_0 \ll -1$ and $C > 0$ so that
  \[
  |v_{\sigma}(\sigma,\tau)| + |v(\sigma,\tau)| \le C \, \delta(\tau)^{\frac18},
  \]
  for $|\sigma| \le \delta(\tau)^{-\theta}$.  Moreover,
  \[
  \delta(\tau)^4
  = \Big ( \sup_{\tau' \le\tau} |v(0,\tau')| \Big)^4
  \le C \sup_{\tau'\leq \tau} \|\bar v(\tau')\|^2.
  \]
\end{lemma}

\begin{proof}
  The first estimate follows by the same proof as of Lemma 3.7 in \cite{Br}.  The second estimate is also shown in \cite{Br} and follows by standard interpolation inequalities.
\end{proof}

We split the proof of Theorem \ref{prop-parabolic-asym} in a few propositions.  First, we derive a more detailed equation for $\bar v$.

\begin{prop}
  The function $\bar v$ satisfies the linear inhomogeneous equation
  \begin{equation}
    \label{eq-vbar-linear-inhomog}
    \bar v_\tau - \cL\bar v
    = a(\sigma,\tau)\bar v_\sigma +
    b(\sigma,\tau)\bar v + c(\sigma,\tau),
  \end{equation}
  where
  \[
  a(\sigma,\tau) = -J[v] +\frac{v_\sigma}{1+v},\qquad b(\sigma,\tau) = - \frac{v}{2(1+v)},
  \]
  and
  \begin{equation}
    \label{eq-c-defined}
    c(\sigma,\tau) =
    \left\{
    - \frac{v v_\sigma}{1+v}
    +J[v]v
    -2 v_\sigma
    \right\} \chi_\sigma
    +\bigl(\chi_\tau - \chi_{\sigma\sigma} + \frac\sigma2 \chi_\sigma\bigr) v \,\,.
  \end{equation}
\end{prop}
\begin{proof}
  In \eqref{eq-bar-v-another} we have an evolution equation for \(\bar v\).
  Using \eqref{eq-v-evol-lin-plus-nonlin} and also $\chi v_\sigma=\bar v_\sigma - \chi_\sigma v$, we then get
  \begin{align*}
    \bar v_\tau - \cL\bar v
    &= \chi\cdot (v_\tau-\cL v) + \bigl(\chi_\tau - \chi_{\sigma\sigma}+\frac\sigma2 \chi_\sigma\bigr) v
    -2\chi_\sigma v_\sigma\\
    &= \frac{\chi v_\sigma^2}{1+v}
    -\frac{\chi v^2}{2(1+v)} -J[v] \, \chi v_\sigma
    +\bigl(\chi_\tau - \chi_{\sigma\sigma} + \frac\sigma2 \chi_\sigma\bigr) v
    -2\chi_\sigma v_\sigma\\
    &= \left(-J[v] +\frac{v_\sigma}{1+v}\right) \bar v_\sigma
    -\frac{v}{2(1+v)} \bar v\\
    &\qquad
    -\frac{\chi_\sigma v v_\sigma}{1+v}
    +J[v] \, \chi_\sigma v
    +\bigl(\chi_\tau - \chi_{\sigma\sigma} + \frac\sigma2 \chi_\sigma\bigr) v
    -2\chi_\sigma v_\sigma.
  \end{align*}
  
\end{proof}

We can apply the variation of constants formula to \eqref{eq-vbar-linear-inhomog}, which tells us that for any given $\tau\ll 0$ one has
\begin{equation}
  \label{eq-VOC}
  \bar v (\tau) = e^{\cL} \bar v(\tau-1)
  + \int_{\tau-1}^\tau e^{(\tau-\tau')\cL}
  \left\{
  a \, \bar v_\sigma(\sigma, \tau') + b \, \bar v(\sigma, \tau')
  + c(\sigma, \tau')
  \right\}
  d\tau'.
\end{equation}
Here $e^{t\cL}$ is the heat semigroup on the Hilbert space $\cH = L^2(\R, d\mu)$ associated with the operator $\cL$.  It has the usual smoothing property \eqref{eq-semigroup-smoothing}, i.e.  ~ it satisfies $\|(1-\cL)^re^{t\cL}\|\leq Ct^{-r}e^t$ for all $t>0$.

\sk

To use \eqref{eq-VOC} we first estimate the coefficients $a$, $b$, and $c$.

\begin{prop}
  Given any $\epsilon>0$ there is a $\tau_0\ll 0$ such that for all $\tau\leq \tau_0$ and $|\sigma| \leq \delta(\tau)^{-\theta}$ one has
  \begin{align*}
    |a(\sigma, \tau)| + |b(\sigma,\tau)| \leq C \, \epsilon\,.
  \end{align*}
  One also has
  \[
  \|c(\cdot, \tau)\| \leq C \, e^{-\frac{1}{64}\delta(\tau)^{-2\theta}}.
  \]
  The derivatives of $a$, $b$, and $c$ satisfy the following estimates
  \[
  |a_\sigma|\leq C, \qquad |b_\sigma|\leq C\epsilon, \quad \text{ and } \quad \|c_\sigma(\cdot, \tau)\| \leq C e^{-\frac{1}{64}\delta(\tau)^{-2\theta}}.
  \]
\end{prop}
\begin{proof}
  Given any $\epsilon>0$, Lemma \ref{lemma-simon-crucial} guarantees the existence of a $\tau_0\ll0$ such that for all $\tau\leq \tau_0$ and $|\sigma|\leq \delta(\tau)$ one has $|v(\sigma,\tau)|+|v_\sigma(\sigma,\tau)|\leq \epsilon$.
  For the nonlocal term $J[v]$ we therefore have
  \[
  |J[v]| \leq \int_0^\sigma \frac{|dv_\sigma|}{1+v} \leq C\epsilon.
  \]
  The coefficient $c(\sigma,\tau)$ is bounded by
  \[
  |c(\sigma,\tau)| \leq C \, |\chi_\sigma| + \big|\chi_\tau - \chi_{\sigma\sigma} + \frac12 \sigma\chi_\sigma\big|
  \]
  Recall that $\chi(\sigma,\tau)=\hat\chi\left(\delta(\tau)^{\theta}\sigma\right)$.  This implies that
  \[
  |\chi_\sigma| \leq C\delta(\tau)^\theta, \qquad 
  |\chi_{\sigma\sigma}| \leq C \delta(\tau)^{2\theta}, \qquad 
  |\sigma\chi_\sigma| \leq C,
  \]
  and also
  \[
  |\chi_\tau(\sigma,\tau)| \leq \theta \, \delta(\tau)^{\theta-1}\delta'(\tau) \, \sigma \, \chi'(\delta(\tau)^\theta \sigma) \leq C \, \delta(\tau)^{-1} \, \delta'(\tau).
  \]
  The hardest term to estimate is $\| \chi_\tau (\cdot, \tau)\|$.  We have
  \[
  \| \chi_\tau (\cdot, \tau)\|^2 \leq C\int_{\frac12\delta(\tau)^{-\theta}}^\infty \delta(\tau)^{-2}\delta'(\tau)^2 e^{-\sigma^2/4} d\sigma \leq C \delta(\tau)^{-1}\delta'(\tau)^2 e^{-\frac1{16}\delta(\tau)^{-2\theta}}
  \]
  which implies
  \[
  \|\chi_\tau (\cdot, \tau)\| \leq C \, \delta(\tau)^{-1/2} e^{-\frac1{32}\delta(\tau)^{-2\theta}} \delta'(\tau).
  \]
  The desired estimate for $ \|\chi_\tau (\cdot, \tau)\|$ now follows from Lemma \ref{prop-delta-Lipschitz} which guarantees that $\delta'(\tau)$ is uniformly bounded.  The rest of the terms in $c(\sigma, \tau)$ are easy to estimate.
  
  \sk
  To estimate the derivatives of the coefficients we use Lemma~\ref{lemma-bound-Q} which tells us that
  \[
  0\leq -u_{\sigma\sigma} \leq \frac{1-u_\sigma^2}{u}\leq \frac 1 u,
  \]
  and hence that $u_{\sigma\sigma}$ and $\bar v_{\sigma\sigma}$ are uniformly bounded for $|\sigma| \leq 2\delta(\tau)^{-\theta}$.
  We have
  \[
  a_\sigma = 2\frac{u_{\sigma\sigma}}{u} + \frac{v_{\sigma\sigma}}{1+v} -\frac{v_\sigma^2}{(1+v)^2},
  \]
  which implies that $a_\sigma$ is uniformly bounded if $|\sigma| \leq 2\delta(\tau)^{-\theta}$.  For $b_\sigma$ we have
  \[
  b_\sigma = \frac{v_\sigma}{2(1+v)^2},
  \]
  which implies $|b_\sigma|\leq C\epsilon$.  Finally, for $c_\sigma$ we differentiate~\eqref{eq-c-defined} with respect to $\sigma$, which leads to many terms, namely
  \begin{multline*}
    \frac{\pd c}{\pd\sigma}= \left\{ \frac{(n-1) v v_\sigma}{1+v} -J[v]v -2 v_\sigma \right\} \chi_{\sigma\sigma} + \chi_{\sigma} \frac{\pd}{\pd\sigma}\left\{ \frac{(n-1) v v_\sigma}{1+v} -J[v]v -2 v_\sigma
    \right\} \\
    +v \, \frac{\pd}{\pd\sigma}\bigl(\chi_\tau - \chi_{\sigma\sigma} + \frac\sigma2 \chi_\sigma\bigr) +\bigl(\chi_\tau - \chi_{\sigma\sigma} + \frac\sigma2 \chi_\sigma\bigr) v_\sigma\ .
  \end{multline*}
  Using $\chi(\sigma,\tau) = \hat\chi(\delta(\tau)^\theta \sigma)$, the boundedness of $\delta'(\tau)$ (Lemma \ref{prop-delta-Lipschitz}), and our bounds for $v$, $v_\sigma$, and $v_{\sigma\sigma}$ we find that $c_\sigma$ is bounded by $C \, \delta(\tau)^{m}$,
  for some constants $C$ and $m$.  Furthermore $c_\sigma$ is supported in the region where $\frac12\delta(\tau)^{-\theta} \leq \sigma \leq \delta(\tau)^{-\theta}$.  This implies the stated estimate for $\|c_\sigma\|$.
  
\end{proof}

\subsection{Spectral decomposition and the dominant mode}
Our linearized operator $\cL$ given by \eqref{eqn-oper} is self-adjoint in the Hilbert Space
\[
\cH= L^2(\R , e^{-\sigma^2/4}\, d\sigma), \qquad
\langle f, g\rangle
= \int_{\R } f(\sigma) g(\sigma) e^{-\sigma^2/4}\, d\sigma.
\]
It satisfies
\[
1-\cL = \pd_\sigma^*\pd_\sigma,
\]
where $\pd_\sigma^*=-\pd_\sigma+\frac{\sigma}{2}$ is the adjoint of $\pd_\sigma$.  Using the identity $[\pd_\sigma^*, \pd_\sigma] = -\frac12$ one finds that
\begin{subequations}
  \label{eq-derivatives-from-L}
  \begin{align}
    &\|f_\sigma\|^2 = \langle f, (1-\cL)f\rangle = \|\sqrt{1-\cL}f\|^2 \\
    &\|f_{\sigma\sigma}\|^2 + \frac12\|f_\sigma\|^2
    = \|(1-\cL)f\|^2
  \end{align}
\end{subequations}

\sk

The operator $\cL$ generates an analytic semigroup on $\cH$ and by the spectral theorem one has the estimates
\begin{equation}
  \label{eq-semigroup-smoothing}
  \|(1-\cL)^re^{t\cL}f\|\leq \frac{C_r}{t^r}e^t
\end{equation}
for all $t>0$.

The operator $\cL$ has a discrete spectrum.  When restricted to reflection symmetric functions, its eigenvalues
are given by $\{\lambda_k\}_{k=0}^{\infty}$, where $\lambda_k = 1 - k$.  The corresponding eigenvectors are the Hermite polynomials $h_{2k}(\sigma)$, where $h_0(\sigma) = 1$, $h_2(\sigma) = \sigma^2 - 2$, etc.  Let us write $\cH= \cH_0 \oplus \cH_+ \oplus \cH_-$, where $\cH_+$ is spanned by $h_0$, $\cH_0$ is spanned by $h_2$, and $\cH_-$ is spanned by the remaining eigenfunctions $\{h_4, h_6, \dots\}$,
and denote by $\mathcal P_+$, $\mathcal P_0$ and $\mathcal P_-$ denote the orthogonal projections associated with the direct sum $\cH= \cH_0 \oplus \cH_+ \oplus \cH_-$.  Also, we define
\[
\bv_\pm(\cdot, \tau) = \mathcal P_\pm\bigl[\bv(\cdot, \tau)\bigr],\quad
\bv_0(\cdot, \tau) = \mathcal P_0\bigl[\bv(\cdot, \tau)\bigr]
\]
so that
\[
\label{eq-proj}
\bv (y,\tau) = \bv_+(y,\tau) + \bv_0(y,\tau) + \bv_-(y,\tau).
\]
It will be convenient to abbreviate
\[
\Gamma(\tau) := \sup_{\tau'\le\tau}\, \|\bar{v}(\cdot,\tau')\|
\]
and similarly
\[
\Gamma_+(\tau) := \sup_{\tau'\leq\tau} \|\bar{v}_+(\tau')\|, \qquad
\Gamma_-(\tau) := \sup_{\tau'\leq\tau} \|\bar{v}_-(\tau')\|, \qquad
\Gamma_0(\tau) := \sup_{\tau'\leq\tau} \|\bar{v}_0(\tau')\|.
\]

\sk
\sk

In Proposition \ref{lemma-prevails-v} we will show that $|\Gamma_{\pm}(\tau)|$ are small compared to the quantity
\begin{equation}
  \label{eq-alpha-star-def}
  \alpha_*(\tau) = \sup_{\tau'\leq \tau} |\alpha(\tau)|, \qquad \alpha(\tau) = \int \bar{v} \, h_2 \, e^{-\sigma^2/4}\, d\sigma.
\end{equation}
and show that the latter dominates for $\tau \ll -1$.  Note that by definition $\alpha_*(\tau)$ is a nondecreasing function of $\tau$ and that
$\Gamma_0(\tau) = \alpha_*(\tau) \, \|h_2\|$ but it is more convenient
in terms of notation to work with $\alpha_*(\tau).  $ The first step to showing this result is the next lemma
where we show that
either $\Gamma_+(\tau)$ dominates or $\Gamma_0(\tau)$ does.

\begin{lemma} Either
  \[
  \Gamma_-(\tau) + \Gamma_0(\tau) = o(\Gamma_+(\tau))
  \qquad (\tau \le\tau_0),
  \]
  or
  \[
  \Gamma_-(\tau) + \Gamma_+(\tau) = o(\Gamma_0(\tau))
  \qquad (\tau \le\tau_0).
  \]
\end{lemma}

\begin{proof}
  Using the variation of constants formula we can represent the solution $\bar{v}(\sigma,\tau)$ to \eqref{eq-bar-v-another} as
  \begin{equation}
    \label{eq-representation}
    \bar{v}(\tau+1) = e^{\mathcal{L}} \bar{v}(\tau) + \int_{\tau}^{\tau+1} e^{(\tau+1-\tau')\, \mathcal{L}}\, E_0(\cdot,\tau')\, d\tau'.
  \end{equation}
  Applying the projections $\mathcal P_\pm$, and using the description of the spectrum of $\mathcal{L}$ given above, we have the following representations for the projections $\bv_\pm$ and $\bv_0$:
  \[
  \begin{aligned}
    &\bar{v}_+(\tau+1) = e\, \bar{v}_+(\tau) + \int_{\tau}^{\tau+1} e^{\tau+1-\tau'}\, \mathcal{P}_+E_0(\cdot,\tau')\, d\tau' \\
    &\bar{v}_-(\tau+1) = e^{\mathcal{L}} \bar{v}_-(\tau) + \int_{\tau}^{\tau+1} e^{(\tau+1-\tau')\, \mathcal{L}}\, \mathcal{P}_- \,
    E_0(\cdot,\tau')\, d\tau'\\
    &\bar{v}_0(\tau+1) = \bar{v}_0(\tau) + \int_{\tau}^{\tau+1} \mathcal{P}_0\, E_0(\cdot,\tau') \, d\tau'.
  \end{aligned}
  \]
  By \eqref{eq-E0-simon} we have
  \begin{equation*}
    \label{eq-cor-simon}
    \begin{split}
      \|\bar{v}_+(\tau+1)\| \ge e\, \|\bar{v}_+(\tau)\| - C \int_{\tau}^{\tau+1} \delta(\tau')^{\frac{\theta}2}\, \|\bar{v}(\tau')\| d\tau' - C \int_{\tau}^{\tau+1} e^{-\frac{1}{16}\delta(\tau')^{-2\theta}} d\tau'.
    \end{split}
  \end{equation*}
  Given $\tau_1 < 0$, choose $\tau_2 \le \tau_1$ so that $\Gamma_+(\tau_1) = \|\bar{v}(\tau_2)\|$.  Then, using that both $\Gamma_+(\tau)$ and $\delta(\tau)$ are nondecreasing functions in time and the last estimate, we obtain
  \begin{align*}
    \Gamma_+(\tau_1+1) &\ge \|\bar{v}_+(\tau_2+1)\| \\
    &\ge e \, \|\bar{v}_+(\tau_2)\| - C \int_{\tau_2}^{\tau_2+1} \delta(\tau')^{\frac \theta2}\, \|\bar{v}(\tau')\| d\tau' - C \int_{\tau_2}^{\tau_2+1} e^{-\frac{1}{16}\delta(\tau')^{-2\theta}} d\tau' \\
    &\ge e\, \Gamma_+(\tau_1) - C\, \delta(\tau_2 + 1)^{\theta}\, \Gamma_+(\tau_2 + 1) - C\, e^{-\frac{1}{16}\delta(\tau_2+1)^{-2\theta}} \\
    &\ge e\, \Gamma_+(\tau_1) - C\, \delta(\tau_1 + 1)^{\theta}\, \Gamma_+(\tau_1 + 1) - C\, e^{-\frac{1}{16}\delta(\tau_1+1)^{-2\theta}}.
  \end{align*}
  This implies that for every $\tau\le \tau_0 \ll -1$ we have
  \begin{equation}
    \label{eq-simon-1}
    \Gamma_+(\tau) \le e^{-1}\, \Gamma_+(\tau+1) + C\, \delta(\tau + 1)^{\theta}\, \Gamma_+(\tau + 1) + C\, e^{-\frac{1}{16}\delta(\tau+1)^{-2\theta}}.
  \end{equation}
  Regarding the negative mode, using \eqref{eq-E0-simon} and the representation for $\bar{v}_-$ we also have that
  \[
  \|\bar{v}_-(\tau+1)\| \le e^{-1}\, \|\bar{v}_-(\tau)\| + C \int_{\tau}^{\tau+1} \delta(\tau')^{\frac \theta2}\, \|\bar{v}(\tau')\|\, d\tau' + \int_{\tau}^{\tau+1} e^{-\frac{1}{16}\delta(\tau')^{-2\theta}}\, d\tau',
  \]
  implying that for all $\tau_1 \leq \tau_0$, we have
  \begin{align*}
    \Gamma_-(\tau_1 + 1) &:= \sup_{\tau \le \tau_1} \|\bar{v}_-(\tau+1)\| \\
    &\le e^{-1} \, \sup_{\tau \le \tau_1} \|\bar{v}_-(\tau)\| + C\, \delta(\tau_1+1)^{\frac \theta2} \Gamma_-(\tau_1 + 1) + C\, e^{-\frac{1}{16}\delta(\tau_1+1)^{-2\theta}}.
  \end{align*}
  Since $\tau_1 \le \tau_0 \ll -1$ is arbitrary, we immediately get that for all $\tau \le \tau_0 \ll -1$ the following holds
  \begin{equation}
    \label{eq-simon-2}
    \Gamma_-(\tau) \ge e\, \Gamma_-(\tau+1) - C\, \delta(\tau+1)^{\frac \theta2} \Gamma_-(\tau + 1) - C\, e^{-\frac{1}{16}\delta(\tau+1)^{-2\theta}}.
  \end{equation}
  Finally, using the representation for $\bar{v}_0$ and \eqref{eq-E0-simon}, similarly to the other two projections we obtain
  \[
  |\Gamma_0(\tau+1) - \Gamma_0(\tau)| \le C\, \delta(\tau+1)^{\frac \theta2} \Gamma_0(\tau + 1) + C\, e^{-\frac{1}{16}\delta(\tau+1)^{-2\theta}}.
  \]
  
  \sk
  \sk
  Next, the key observation (as in \cite{Br}) is that
  \begin{equation}
    \label{eq-key-simon}
    \delta(\tau)^4 \le C \, \Gamma(\tau),
  \end{equation}
  by standard interpolation inequalities.  Consequently,
  \[
  e^{-\frac{1}{16}\delta(\tau+1)^{-2\theta}} \le C\, \delta(\tau)^5 \le C\, \delta(\tau)\, \Gamma(\tau).
  \]
  Thus, we conclude that
  \begin{equation}
    \label{eq-first-ineq-simon}
    \begin{split}
      & \Gamma_+(\tau) \le e^{-1} \Gamma_+(\tau+1) + C\, \delta(\tau+1)^{\theta/2}\, \Gamma_+(\tau+1) \\
      & |\Gamma_0(\tau+1) - \Gamma_0(\tau)| \le C\, \delta(\tau+1)^{\theta/2}\, \Gamma_0(\tau+1)\\
      & \Gamma_-(\tau) \ge e\, \Gamma_-(\tau+1) - C\, \delta(\tau+1)^{\theta/2}\, \Gamma_-(\tau+1).
    \end{split}
  \end{equation}
  The rest of the proof is identical to the proof of Lemma 3.9 in \cite{Br}.
\end{proof}

We will next show the crucial for our purposes result which states that $\Gamma_0(\tau)$ actually dominates for $\tau \ll -1$.

\begin{proposition} \label{lemma-prevails-v}
  There exists a $\tau_0 \ll -1$ so that for all $\tau \le \tau_0$ we have
  \[
  \Gamma_+(\tau) + \Gamma_-(\tau) = o(\Gamma_0(\tau)).
  \]
  
\end{proposition}

\begin{proof}
  We will assume that $\Gamma_0(\tau) + \Gamma_-(\tau) = o \big (\Gamma_+(\tau)\big )$ for all $\tau \le \tau_0$ and obtain contradiction.
  Some of the estimates here are similar to the ones we will use in the proof of Theorem \ref{prop-parabolic-asym} that follows next.
  We refer the reader to that proof for these details.
  \sk
  
  Under the assumption of reflection symmetry, $\lambda_+=1$ is the only positive eigenvalue and hence, similarly to obtaining \eqref{eq-simon-1} we can also get that
  \begin{equation}
    \label{eq-iterate-simon}
    \Gamma_+(\tau) \ge e^{-1} \Gamma_+(\tau+1) - C\, \delta(\tau+1)^{\theta/2}\, \Gamma_+(\tau+1).
  \end{equation}
  If we iterate \eqref{eq-first-ineq-simon} and \eqref{eq-iterate-simon}, like in \cite{Br} we obtain
  \[
  O(e^{(1+\epsilon)\tau}) \le \Gamma_+(\tau) \le O(e^{(1-\epsilon)\tau}),
  \]
  for every $\epsilon > 0$.  This together with \eqref{eq-key-simon} imply
  $\delta(\tau) \le e^{\frac{\tau}{8}}$.  Using this and iterating \eqref{eq-first-ineq-simon} and \eqref{eq-iterate-simon} again yield
  \begin{equation}
    \label{eq-positive-rate}
    \Gamma_+(\tau) = O(e^{\tau}) \qquad \mbox{and thus} \qquad \Gamma(\tau) = O(e^{\tau}),
  \end{equation}
  by our assumption that $\Gamma_0(\tau) + \Gamma_-(\tau) \le o(\Gamma_+(\tau))$.
  
  \sk
  
  Recall that the eigenfunction corresponding to positive eigenvalue is $h_0 = 1$.  Multiply \eqref{eq-bar-v} by $h_0$ and integrate over $\mathbb{R}$ to get
  \[
  \beta_{\tau} = \beta + \int \Big(E + E_{\chi} + E_{nl} + \frac{1}{100} \,v \chi' \sigma \, \delta(\tau)^{-\frac{99}{100}} \delta'(\tau)\Big)\, d\mu.
  \]
  Under our assumption about the prevailing mode, following same arguments as in the proof of Proposition \ref{prop-error-alpha}
  below, we get
  \[
  \left|\int \Big(E + E_{\chi} + E_{nl}\Big)\, d\mu\right| = O\Big(\beta_*^2(\tau)\Big),
  \]
  where $\beta_*(\tau) = \sup_{\tau' \le\tau}|\beta(\tau')|$.  Similar arguments that we will use in the proof of Theorem \ref{prop-parabolic-asym} to derive differential inequalities for $\alpha(\tau)$, yield also that
  \[
  \beta_{\tau} = \beta + O\Big(\beta_*^2(\tau)\Big).
  \]
  Following the same arguments as in the proof of Theorem \ref{prop-parabolic-asym} we can show there exists a $\tau_0 \ll -1$ so that for $\tau \le \tau_0$,
  \[
  \beta_{\tau} = \beta + O\Big(\beta^2(\tau)\Big).
  \]
  By \eqref{eq-positive-rate} we have $|\beta(\tau)| = O(e^{\tau})$ and hence,
  \[
  \beta_{\tau} = \beta\, \Big(1 + O(e^{\tau})\Big).
  \]
  This implies
  \[
  \beta(\tau) = (K + O(e^{\tau}))\, e^{\tau}.
  \]
  We can argue as in the proof of Lemma 5.8 in \cite{ADS1} to claim that $K$ can not vanish.  The same reparametrization arguments that we used in the proof of Lemma 5.11 in \cite{ADS1} yield a contradiction.
  This concludes the proof of Proposition \ref{lemma-prevails-v}.
\end{proof}

\subsection{The proof of Theorem \ref{prop-parabolic-asym}} Using Proposition \ref{lemma-prevails-v} we are able to claim first the following result on derivatives of $\bar{v}$.

\begin{lemma}
  \label{lemma-Sigurd}
  There exists a $\tau_0 \ll -1$ so that for all $\tau \le \tau_0$ we have
  \[
  \|(1-\cL)^r\bar v(\tau)\|
  = O(\alpha_*(\tau))\qquad (\tau\to-\infty).
  \]
  for $0\leq r<\frac32$.  In particular, for $r=\frac12$ and $r=1$ this implies
  \[
  \|\bar v_{\pm,\sigma}(\tau)\| + \|\bar v_{\pm,\sigma\sigma}(\tau)\|
  = o(\alpha_*(\tau))
  \qquad (\tau\to-\infty),
  \]
  where $\bar{v}_{\pm}(\sigma,\tau) := \bar{v}_+(\sigma,\tau) + \bar{v}_-(\sigma,\tau)$.
\end{lemma}

At this point we can estimate the first derivative of $\bar v(\tau)$.  After that, we can bootstrap and estimate $\|(1-\cL)^r\pd_\sigma \bar v(\tau)\|$.
\begin{lemma}
  There is a constant $C$ such that for all $\tau\ll0$ one has
  \[
  \sup_{\tau'\leq\tau}\|\bar v_\sigma(\tau')\| \leq C \alpha_*(\tau).
  \]
\end{lemma}
\begin{proof}
  
  We consider the two terms in \eqref{eq-VOC}.  The first term satisfies
  \[
  \left\|\pd_\sigma e^{\cL}\bar v(\tau-1)\right\| \leq C\|\bar v(\tau-1)\| \leq C\alpha_*(\tau).
  \]
  Thus \eqref{eq-VOC} implies
  \begin{align*}
    \|\bar v_\sigma(\tau)\|
    &\leq C\alpha_*(\tau) + C \epsilon
    \int_{\tau-1}^\tau
    \frac{\|\bar v_\sigma(\tau')\| + \|\bar v(\tau')\| }{\sqrt{\tau-\tau'}} d\tau'
    + C\int_{\tau-1}^\tau
    \frac{e^{-\frac{1}{64}\delta(\tau')^{-2\theta}}}{\sqrt{\tau-\tau'}} d\tau' \\
    &\leq C\alpha_*(\tau) +Ce^{-\frac{1}{64}\delta(\tau)^{-2\theta}}
    + C\epsilon \sup_{\tau'\leq \tau} \|\bar v_\sigma(\tau')\|
  \end{align*}
  If we choose $\epsilon>0$ so small that $C\epsilon\leq \frac12$ then, after taking the supremum over $\tau$ and using the fact that $\delta(\tau)$ is nondecreasing, we get
  \[
  \sup_{\tau'\leq \tau}\|\pd_\sigma \bar v(\tau')\| \leq C\alpha_*(\tau) + C e^{-\frac{1}{64}\delta(\tau)^{-2\theta}}.
  \]
  Finally, we use that for all sufficiently large $\tau$ Lemma \ref{lemma-simon-crucial} implies that
  \[
  e^{-\frac{1}{64}\delta(\tau)^{-2\theta}} \leq C_m\delta(\tau)^m\sup_{\tau'\leq \tau} \|\pd_\sigma\bar v(\tau')\|
  \]
  so that for $\tau\ll0$ we get
  \[
  \sup_{\tau '\leq\tau}\|\pd_\sigma\bar v(\tau')\| \leq C \alpha_*(\tau),
  \]
  as claimed.
\end{proof}

\begin{lemma}
  There is a constant $C$ such that for all $\tau\ll0$ one has for any $r<1$
  \[
  \sup_{\tau'\leq\tau} \|(1-\cL)^r\bar v_\sigma(\tau')\| \leq C\alpha_*(\tau).
  \]
  In particular, setting $r=\frac12$ leads to
  \[
  \sup_{\tau'\leq\tau}\|\bar v_{\sigma\sigma}(\tau')\| \leq C \alpha_*(\tau).
  \]
  
\end{lemma}
\begin{proof}
  We consider $w=\bar v_\sigma$ and obtain an equation for $w$ by differentiating~\eqref{eq-vbar-linear-inhomog} with respect to $\sigma$.  We get
  \[
  w_\tau - \cL w = \left(a_\sigma-\tfrac12+b\right) w + aw_\sigma + b_\sigma \bar v + c_\sigma \ .
  \]
  Once again we can apply the variation of constants formula to get
  \[
  w(\tau)= e^{\cL}w(\tau-1) + \int_{\tau-1}^\tau e^{(\tau-\tau')\cL} \left\{ \left(a_\sigma-\tfrac12+b\right) w + aw_\sigma + b_\sigma \bar v + c_\sigma \right\} d\tau'.
  \]
  We now estimate the norm of $(1-\cL)^rw$.  By choosing $r>\frac12$ we ensure that this norm will be better than $\|w_\sigma\|$, while choosing $r<1$ leads to convergent integrals in the following application of the variation of constants formula.  Thus we get
  \begin{multline*}
    \|(1-\cL)^rw(\tau)\| \leq
    C\|w(\tau-1)\| +\\
    C\int_{\tau'-1}^\tau \frac{\|w(\tau')\|+ \epsilon\|w_\sigma(\tau')\| + \|\bar v(\tau')\| + e^{-\frac{1}{64}\delta(\tau)^{-2\theta}} } {(\tau-\tau')^r} d\tau'\ .
  \end{multline*}
  Estimate $\|w_\sigma(\tau')\|\leq C \|(1-\cL)^rw(\tau')\|$, and then proceed as before to conclude $ \|(1-\cL)^rw(\tau)\| \leq C \alpha_*(\tau) $.
  
\end{proof}

We are ready now to provide the proof of Lemma \ref{lemma-Sigurd}.

\begin{proof}[Proof of Lemma \ref{lemma-Sigurd}]
  At this point we have shown that derivatives up to order almost 3 of $\bar v(\tau)$ are bounded in terms of $\alpha_*(\tau)$.
  To prove Lemma~\ref{lemma-Sigurd} we have to show that $\bar v_{\pm}(\tau)=\bar v_+(\tau)+\bar v_-(\tau)$ and its first and second order derivatives are $o(\alpha_*(\tau))$.
  This follows from an interpolation argument.
  Namely, by Proposition \ref{lemma-prevails-v} we know that $\|\bar v_{\pm}(\tau)\| = o(\alpha_*(\tau))$, and we have shown that $\|(1-\cL)^r\bar v\| = O(\alpha_*(\tau))$ for any $r<\frac32$.
  The interpolation inequality
  \[
  \|(1-\cL)^s \phi\| \leq \|\phi\|^{1-s/r}\|(1-\cL)^r\phi\|^{s/r}
  \qquad (0\leq s\leq r)
  \]
  implies that $\|(1-\cL)^r\bar v(\tau)\| = o(\alpha_*(\tau))$ for $r\in \{\frac12, 1\}$.
  In view of \eqref{eq-derivatives-from-L}, this implies that $\|\pd_\sigma^m\bar v_{\pm}(\tau)\| = o(\alpha_*(\tau))$ for $m \in \{1,2\}$, and therefore also completes the proof of Lemma~\ref{lemma-Sigurd}.
\end{proof}

\subsection{Asymptotics of $\alpha(\tau)$}
It is easy to see that $v(\sigma,\tau)$ satisfies
\[
v_{\tau} =
v_{\sigma\sigma} - \frac{\sigma}{2}\, v_{\sigma} + v - v_{\sigma}^2 - \frac{v^2}{2} + \frac{v v_{\sigma}^2}{1+v} + \frac{v^3}{2(1+v)} - 2 v_{\sigma}\, \int_0^{\sigma} \frac{v_{\sigma}^2d\sigma}{(1+v)^2}.
\]
We also have
\begin{equation}
  \label{eq-bar-v}
  \bar{v}_{\tau} = \bar{v}_{\sigma\sigma} - \frac{\sigma}{2}\bar{v}_{\sigma} + \bar{v} - \bar{v}_{\sigma}^2 -\frac{\bar{v}^2}{2} + E + E_{\chi} + E_{nl}
\end{equation}
holding for almost all $\tau$ (where $\delta'(\tau)$ exits), and
\begin{equation}
  \label{eq-E-bar-v}
  E := \frac{\bar{v}\bar{v}_{\sigma}^2}{1+v} + \frac{\bar{v}^3}{2(1+v)},
\end{equation}
is the error term containing at least quadratic and higher order terms in $\bar{v}$ and its derivatives,
\begin{equation}
  \label{eq-E-chi}
  \begin{split}
    E_{\chi}
    &:= -2v_{\sigma} \chi_\sigma
    - v\chi_{\sigma\sigma}
    + \frac{\sigma}{2} v \chi_\sigma
    - (1- \chi)v_{\sigma}\bar{v}_{\sigma}
    + \chi_\sigma v \bar{v}_{\sigma}
    + \chi_\sigma v_{\sigma}^2 \\
    &- \frac{v\bar{v}(1- \chi)}{2}
    + \frac{\bar{v}(1- \chi)^2v_{\sigma}^2}{1+v}
    + \frac{\bar{v}\chi_\sigma^2v^2}{1+v}
    + \frac{\bar{v} v^2 (1- \chi)^2}{2(1+v)}
    +\frac{\bar{v}^2 v(1- \chi)}{1+v} \\
    &+ \frac{2\bar{v}}{1+v}\, \Big(\bar{v}_{\sigma} v_{\sigma} (1-\chi)
    - v \bar{v}_{\sigma}\chi_\sigma
    - \chi_\sigma (1- \chi) v v_{\sigma} \Big) + v \, \chi_\tau
  \end{split}
\end{equation}
is the error term coming from introducing the cut off function $\chi$ and
\begin{equation}
  \label{eq-E-NL}
  E_{nl} := -2 \chi \, v_{\sigma}
  \int_0^{\sigma} \frac{v_{\sigma}^2d\sigma}{(1+v)^2},
\end{equation}
is the non-local error term.

In order to analyze the behavior of $\alpha(\tau)$ we will multiply equation \eqref{eq-bar-v} by $h_2$ and integrate it over $\R $.  We get
\[
\alpha_{\tau} =
- \int \left(\frac{\bar{v}^2}{2} + \bar{v}_{\sigma}^2\right) h_2\, d\mu
+ \int \left(E + E_{\chi} + E_{nl}\right) h_2\, d\mu.
\]
We claim the first integral in the ODE for $\alpha(\tau)$ above is the leading order term, while the others are of lower order.  We verify this in the following proposition.
\sk

\begin{proposition}\label{prop-error-alpha}
  There exists a $\tau_0 \ll -1$ so that for all $\tau \le \tau_0$ we have
  \[
  \int \left(\frac{\bar{v}^2}{2} + \bar{v}_{\sigma}^2\right) h_2\, d\mu
  = 8\alpha(\tau)^2 \, \|h_2\|^2
  + o\bigl(\alpha_*(\tau)^2\bigr),
  \]
  and
  \[
  \left|\int E\, h_2\, d\mu\right|
  + \left|\int E_{\chi} h_2\, d\mu\right|
  + \left|\int E_{nl} h_2\, d\mu\right|
  = o\bigl(\alpha_*(\tau)^2\bigr),
  \]
  where the right hand sides in both integrals above are evaluated at time $\tau$.
\end{proposition}

\begin{proof}
  Let ${\bar v}_\pm := \bv_+ + \bv_- =\bar{v} - \alpha h_2$.  We can write
  \begin{equation*}
    \begin{split}
      \int \left(\frac{\bar{v}^2}{2} + \bar{v}_{\sigma}^2\right) h_2\, d\mu
      =& \alpha(\tau)^2 \int \big ( \frac{h_2^3}2
      + h_{2\sigma}^2 h_2 \big ) \, d\mu \\
      &+ 2\alpha(\tau) \int \big ( h_2^2 \bv_\pm + (\bv_\pm)_{\sigma} h_{2\sigma} h_2\, \big ) \, d\mu \\
      &+\int \big ( \frac{\bv_\pm^2}{2} +
      ( \bv_\pm)_\sigma^2 \big ) \, h_2\, d\mu
    \end{split}
  \end{equation*}
  Using $h_2^2 = h_4 + 8h_2 + 8h_0$ and $h_{2\sigma}^2 = 4h_2 + 8h_0$ we compute
  \[
  \int \left(\frac{h_2^3}{2} + h_{2\sigma}^2h_2\right)\, d\mu = 8\, \|h_2\|^2.
  \]
  To estimate the other terms, we recall Lemma 4.12 in \cite{ADS1} which implies that there exist universal constants $C_1, C_2$ so that if a function $f$ is compactly supported in $\R $, then
  \begin{equation}
    \label{eq-old-stuff}
    \int f^2\sigma^2\, d\mu \le C_1 \int f^2\, d\mu + C_2 \int f_{\sigma}^2\, d\mu.
  \end{equation}
  Using this yields
  \[
  \left|\int \bv_\pm^2 h_2 \, d\mu\right|
  \le C_1 \int \bv_\pm^2\, d\mu
  + C_2 \int(\bv_\pm)_{\sigma}^2\, d\mu.
  \]
  Combining the last estimate with Lemma \ref{lemma-Sigurd} yields
  \[
  \left|\int \bv_\pm^2h_2\, d\mu \right|
  \le C \, \sup_{\tau'\leq \tau} \Big ( \|\bv_\pm(\tau)\|^2
  + \|(\bv_\pm)_{\sigma}(\tau) \|^2 \Big )
  = o(\alpha_*(\tau)^2).
  \]
  Similarly,
  \[
  \left|\int(\bar{v}_\pm )_{\sigma}^2 h_2\, d\mu \right| = o(\alpha_*(\tau)^2).
  \]
  By Cauchy-Schwartz
  \[
  \int \bar{v}_\pm h_2^2\, d\mu = o(\alpha_*(\tau)) \quad \mbox{and} \quad
  \int (\bv_\pm)_{\sigma} h_{2\sigma} h_2\, d\mu = o(\alpha_*(\tau)).
  \]
  By putting everything above together we obtain
  \[
  \int \left(\frac{\bar{v}^2}{2} + \bar{v}_{\sigma}^2\right) h_2\, d\mu
  = 8\alpha(\tau)^2\, \|h_2\|^2 + o(\alpha_*(\tau)^2),
  \]
  as stated in the first part of the Proposition.
  
  To show the second estimate in the statement of the Proposition let us start with $E$ given in \eqref{eq-E-bar-v}.  Recall
  that by Lemma \ref{lemma-simon-crucial}, we have
  \begin{equation}
    \label{eq-simon0}
    |v_{\sigma}(\sigma,\tau)| + |v(\sigma,\tau)| \le C \delta(\tau)^{\frac 18}, \qquad \mbox{for} \,\,\,\, |\sigma| \le \delta(\tau)^{-\theta}.
  \end{equation}
  This estimate and Lemma \ref{lemma-Sigurd} yield the bound
  \begin{equation}\label{eq-error1}
    \big | \int E \, h_2 d\mu \big | \leq C\, \sup |\bv | \, \int \Big ( \bar{v}_{\sigma}^2 + \frac{\bar{v}^2}{2} \Big ) (\sigma^2 +2)\, d\mu
    = o(\alpha_*(\tau)^2).
  \end{equation}
  \sk
  
  Any of the terms in \eqref{eq-E-chi} is supported in $|\sigma| \ge \delta(\tau)^{-\theta}$ and hence integral of any of them is exponentially small.  That is why we treat all terms in \eqref{eq-E-chi} the same way, so we will discuss in details only one of them.  Using that $|v(\sigma,\tau)| \le C \delta(\tau)^{\frac18}$ (Lemma 3.7 in \cite{Br}) and the key observation in \cite{Br} that
  \begin{equation}\label{eq-simon}
    \delta(\tau)^4 \le C \sup_{\tau'\leq \tau}\|v(\tau')\|^2 \le C\alpha_*(\tau)^2
  \end{equation}
  (that follows by standard interpolation inequalities) we have $|v(\sigma,\tau) \delta(\tau)^\theta h_2| \le C$ and for $\tau \le \tau_0 \ll -1$ sufficiently small, we obtain
  \begin{equation}
    \label{eq-error-chi-eg}
    \big | \int v \chi_\tau h_2\, d\mu\big | 
    = \big |\int \sigma v \hat\chi' \delta(\tau)^\theta h_2\, d\mu\big |
    \le C e^{-\frac14\,\delta(\tau)^{-2\theta}} 
    \le C \delta(\tau)^{10} 
    = o(\alpha_*(\tau)^2).
  \end{equation}
  All other estimates are similar, hence concluding
  \begin{equation}
    \label{eq-error2}
    \big |\int E_{\chi} h_2\, d\mu\big | = o(\alpha_*(\tau)^2).
  \end{equation}
  Finally, to treat the term $E_{nl}$, we express $v_{\sigma} = \bar{v}_{\sigma} + (1- \chi) v_{\sigma} - v \chi_\sigma$.  Then, using
  that $\chi_\sigma= \delta(\tau)^\theta \, \hat \chi'$ and also that $v_{\sigma} \le 0$ for $\sigma \ge 0$ and $v_{\sigma\sigma} \le 0$, we obtain
  \begin{align*}
    \big |\int E_{nl} h_2\, d\mu \big | &\le C\, \int |v_{\sigma}|^3 (\sigma^3 + 1)\, d\mu \\
    & \le C\, \int \big (|\bar{v}_{\sigma}| + (1- \chi) |v_{\sigma}| + |v \hat \chi' |\delta(\tau)^\theta\big )^3 (\sigma^3 + 1)\, d\mu.
  \end{align*}
  Once we expand the last integral on the right hand side, note that all integrals of terms that contain either $\hat\chi'$ or $1-\hat\chi$ (since those functions are supported on a set $|\sigma| \ge \delta(\tau)^{-\theta}$) can be estimated by $o(\alpha_*(\tau)^2)$, by the same reasoning as in \eqref{eq-error-chi-eg}.  On the other hand, using \eqref{eq-simon0}, we have
  \[
  \int |\bar{v}_{\sigma}|^3 (\sigma^3 + 1)\, d\mu 
  \le C\delta^{-3\theta + \frac18} \sup_{s\le\tau} \int \bar{v}_{\sigma}^2 \, d\mu 
  = o(\alpha_*(\tau)^2),
  \]
  since from our asymptotics we know $\sup_{\tau'\leq\tau} \|\bar{v}_{\sigma}(\tau')\| = O(\alpha_*(\tau))$.  Hence,
  \begin{equation}
    \label{eq-error3}
    \big |\int E_{nl} h_2\, d\mu\big | = o(\alpha_*(\tau)^2).
  \end{equation}
  Finally, estimates \eqref{eq-error1}, \eqref{eq-error2} and \eqref{eq-error3} conclude the proof of the Proposition.
\end{proof}

\sk
\sk
\subsection{The conclusion of the proof of Theorem \ref{prop-parabolic-asym}} Using the estimates we just obtained,
we will now conclude the proof of our parabolic region asymptotics.

\sk

\begin{proof}[Proof of Theorem \ref{prop-parabolic-asym}] 
  Multiply \eqref{eq-bar-v} by $h_2$ and integrate it over $\R $ with respect to measure $d\mu$.  By Proposition \ref{prop-error-alpha} we get
  \[
  \label{eq-alphadelta}
  \|h_2\|^2 \, \alpha'(\tau) = -8 \, \alpha(\tau)^2 \, \|h_2\|^2 + o(\alpha_*(\tau)^2) + \theta\, \delta'(\tau) \delta(\tau)^{\theta-1} \int v \chi' \sigma\, d\mu.
  \]
  Since \(\delta'(\tau)\geq 0\) is bounded, this implies
  \begin{equation*}
    \alpha'(\tau) 
    \le -8\, \alpha(\tau)^2 + o(\alpha_*(\tau)^2) 
    + C \delta(\tau)^{\theta-1+\frac18} e^{-\frac14 \delta(\tau)^{-2\theta}} 
  \end{equation*}
  Using Lemma \ref{prop-delta-Lipschitz}, Proposition \ref{lemma-prevails-v} and the key observation \eqref{eq-simon} from \cite{Br},
  namely that $\delta(\tau)^4 \le C \, \Gamma(\tau)$, we find that the last term is also $o(\alpha_*(\tau)^2)$. Thus, we conclude that
  \begin{equation}\label{eqn-a100}
    \alpha'(\tau) \le -8\, \alpha(\tau)^2 + o\bigl(\alpha_*(\tau)^2\bigr).
  \end{equation}
  
  We next claim that for $\tau \le \tau_0 \ll -1$, we have
  \begin{equation}\label{eqn-alpha10}
    |\alpha(\tau)| = \alpha_*(\tau), \qquad \mbox{implying that} \quad o\big(\alpha_*(\tau)^2\big) = o(\alpha(\tau)^2).
  \end{equation}
  In order to see that, it is sufficient to show that for $\tau_0\ll -1$ sufficiently negative, {\em $|\alpha(\tau)|$ is monotone increasing for $\tau \le \tau_0$.}
  We argue by contradiction.  If this were not true, since $\lim_{\tau\to -\infty} \alpha(\tau) = 0$, we would be able to find a decreasing sequence of times $\tau_j \to -\infty$ so that each $\tau_j$ is a local maximum for $|\alpha(\tau)|$ and $|\alpha(\tau_j)| = \sup_{\tau'\le\tau_j} |\alpha(\tau')|
  =\alpha_*(\tau_j) >0$.  Since $\alpha(\tau_j) \neq 0$ and $\tau_j$ is local maximum for $|\alpha(\tau_j)|$, we
  must have $\alpha'(\tau_j)=0$, and by \eqref{eqn-a100}
  \[
  8\, \alpha(\tau_j)^2
  = o\bigl(\alpha_*(\tau_j )^2\bigr)
  = o \bigl(\alpha(\tau_j )^2\bigr)
  \qquad
  (j\to\infty),
  \]
  which contradicts $|\alpha(\tau_j) | >0$.  We therefore conclude that \eqref{eqn-alpha10} holds.  
  
  Using what we have just shown, we conclude from \eqref{eqn-a100} that for $\tau \le \tau_0 \ll -1$ we have
  \[
  \alpha'(\tau) \le - 8 \, \alpha(\tau)^2 + o \big (\alpha(\tau)^2 \big ).
  \]
  Integration of this differential inequality yields
  \begin{equation}
    \label{eq-1alpha1}
    \frac{1}{\alpha(\tau)} \le -8 \, |\tau| \, \big (1+o(1) \big ).
  \end{equation}
  Similarly as above, \eqref{eq-alphadelta} also yields
  \[
  \alpha'(\tau) \ge -8\, \alpha(\tau)^2 - o(\alpha_*(\tau)^2)
  \]
  which after integration yields
  \begin{equation}
    \label{eq-1alpha2}
    \frac{1}{\alpha(\tau)} \ge -8 \, |\tau| \, \big (1+o(1) \big ).
  \end{equation}
  Finally, \eqref{eq-1alpha1} and \eqref{eq-1alpha2} imply
  \[
  \alpha(\tau) = -\frac{1}{8|\tau|}\, (1 + o(1)).
  \]
  This concludes the proof of Theorem \ref{prop-parabolic-asym}.
\end{proof}

\section{Intermediate region asymptotics}

\label{section-intermediate}

We will assume in this section that our solution $u(\sigma,\tau) $ of \eqref{eqn-un} satisfies the parabolic region asymptotics \eqref{eqn-u-asym-p}.
Our goal is to derive the asymptotic behavior of $u(\sigma,\tau)$, as $\tau \to -\infty$, in the {\em intermediate region}
\[
\cI_{M,\theta} := \big \{ \, (\sigma, \tau) :\,\, |\sigma| \geq M \quad \mbox{and} \quad u(\sigma,\tau) \geq \ \theta \, \big \}.
\]
To this end, we consider the change of variables $z:=\sigma/\sqrt{|\tau|}$ defining the function
\begin{equation}\label{eqn-dfn-uz}
  \bar u(z,\tau) = u(\sigma, \tau), \qquad z:=\frac{\sigma}{\sqrt{|\tau|}}.
\end{equation}
A direct computation starting from \eqref{eqn-u0} shows that $\bar u(z,\tau)$ satisfies the equation
\begin{equation}\label{eqn-uz}
  \bar u_\tau = \frac 1{|\tau|} \bar u_{zz} - \frac{z}{2} \, \bar u_{z}
  - \frac 1{|\tau|} \bar J(z,\tau) \, {\bar u}_{z} + \frac{{\bar u}_{z}^2}{|\tau| u} - \frac 1{\bar u} + \frac {\bar u}2.
\end{equation}
where
\begin{equation}\label{eqn-J25}
  \bar J(z,\tau):= 2 \, \int_0^z \frac{{\bar u}_{zz}}{u}\, dz.
\end{equation}
Our goal is to prove the following result.

\begin{prop}\label{prop-intermediate} 
  For any $\theta \in (0, \sqrt{2})$, we have
  \begin{equation}\label{eqn-asymz1}
    \lim_{\tau \to -\infty} \bar u(z,\tau) = \sqrt{2 - \frac{z^2}{2}}
  \end{equation}
  uniformly for $u \geq \theta$.  
  It follows that \eqref{eqn-asymz1} holds uniformly on $|z| \leq 2-\eta$,
  for any $\eta >0$ small.  In particular, we have
  \[
  \sigma(\tau) = 2\sqrt{|\tau|} (1 +o(1)), \qquad \mbox{as} \qquad \tau\to -\infty.
  \]
\end{prop}

\sk

Recall the change of variable $Y(u,\tau):=u_\sigma^2(\sigma,\tau)$, $u=u(\sigma,\tau)$ which was introduced in Section \ref{sec-equations}.  We begin by observing that our asymptotics \eqref{eqn-u-asym-p} imply the following asymptotics for $Y(u,\tau)$.

\begin{lemma}
  \label{lemma-Y10} 
  For any $M \gg 1$ we have
  \begin{equation}\label{eqn-YM1}
    Y(u(\sigma, \tau), \tau) = \frac 1{2|\tau|} \big (u(\sigma,\tau)^{-2} - 1 \big ) + \frac 1{4|\tau|^2} + o_M \big (\frac 1{|\tau|^2} \big )
  \end{equation}
  on $|\sigma| \leq M$.
\end{lemma}

\begin{proof} We use the asymptotics \eqref{eqn-u-asym-p} which imply that on $|\sigma| \leq M$,
  \[
  u_\sigma^2 = \frac {\sigma^2}{8\, |\tau|^2} \big ( 1 + o_M (1 ) \big ).
  \]
  On the other hand, using also \eqref{eqn-u-asym-p} we have
  \begin{equation}\label{eqn-u13}
    u^2 = 2\, \big ( 1 - \frac{\sigma^2-2}{4|\tau|} + o_M( \frac 1{|\tau|}) \big )
  \end{equation}
  implying that
  \begin{equation}\label{eqn-u2-3}
    \frac 2{u^2} - 1 = \frac{\sigma^2-2}{4|\tau|} + o_M( \frac 1{|\tau|}).
  \end{equation}
  Thus, at $u=u(\sigma, \tau)$, we have
  \[
  u_\sigma^2 = \frac {\sigma^2}{8\, |\tau|^2} \big ( 1 + o_M (1 ) \big ) =
  \frac 1{2|\tau|} \big ( \frac 2{u^2} - 1 \big ) + \frac 1{4|\tau|^2} + o_M \big (\frac 1{|\tau|^2} \big ).
  \]
  
\end{proof}

\sk

The next result follows by using the barriers constructed by S.  Brendle in \cite{Br} (see Proposition \ref{prop-barriers}) and the exact behavior of $Y(u,\tau)$ at $u(M,\tau)$
shown in \eqref{eqn-YM1}.

\begin{prop}\label{prop-Y-bound-above}Assume that our solution $u(\sigma,\tau)$ satisfies the asymptotics \eqref{eqn-u-asym-p} in the parabolic region $ |\sigma| \leq M$, for all $M \gg 1$, implying by Lemma \ref{lemma-Y10}, that $Y(u,\tau):=u_\sigma^2(\sigma,\tau)$ with $u=u(\sigma,\tau)$ satisfies
  \[
  Y(u,\tau) = \frac 1{2|\tau|} \big ( 2 u^{-2} - 1 \big ) + o \big (\frac 1{|\tau|} \big ) \qquad \mbox{for} \,\, u = u(M, \tau) < \sqrt{2}
  \]
  and all $\tau \leq \tau_0(M) \ll -1$.
  Then, for any $\theta \in (0,1)$ and $\delta >0$, there exists $M_\delta>0$ depending on $\delta$ and $\tau_0 = \tau_0(\delta, \theta) \ll -1$ such that for all $\tau \leq \tau_0 \ll -1$, we have
  \begin{equation}\label{eqn-Y15}
    u_\sigma^2(\sigma,\tau) = Y(u,\tau) \leq \frac {1+\delta}{2|\tau|} \big ( 2u^{-2} - 1 \big ), \qquad \mbox{for} \,\, \, \theta \leq u \leq u(M_\delta, \tau).
  \end{equation}
\end{prop}

\smallskip

\begin{proof}[Proof of Proposition \ref{prop-Y-bound-above}] Fix any small numbers $\delta >0$ and $\theta >0$ and for
  $\tau \leq \tau_0 = \tau_0(\delta, \theta)$ to be determined later, consider the barrier $Y_a (u)$
  given in Proposition \ref{prop-barriers} with $a=a(\tau)$ satisfying
  ${\ds a^{-2}= \frac{1+\delta}{2|\tau|}.}$ For $M =M(\delta) \gg 1$ to be determined later set $\lambda(\tau) = u(M,\tau)$.  Our result will follow by
  comparing $Y(u,\tau)$ with $Y_a(u)$ using the maximum principle on the set
  \begin{equation}\label{eqn-Qtau}
    Q_{\tau} := \bigcup_{{\large \bar \tau \leq \tau' \leq \tau } } \Big ( \big [ a^{-1} r^*, \lambda(\tau) \big ] \times \{ \tau' \} \Big ),
  \end{equation}
  for $\bar \tau \ll \tau \leq \tau_0$.  Here $r^*$ is the universal constant as in Proposition \ref{prop-barriers}.
  
  Lets start by showing that there exists $\tau_0 = \tau_0(\delta,\theta) \ll -1$ such that for all $\tau \leq \tau_0$, we have
  \[
  Y(\lambda(\tau'), \tau') \leq Y_a ( \lambda(\tau') ), \qquad \,\,\, \mbox{for}\,\, \tau' \leq \tau.
  \]
  To this end we will combine \eqref{eqn-YM1} with \eqref{eqn-asym-Ya} from which we get the lower
  bound
  \[
  Y_a(u) \geq a^{-2} ( 2 u^{-2} -1 ) + \frac 1{100} \, a^{-4}
  \]
  (see Proposition 2.9 in \cite{Br} for this bound).
  Using this inequality we find that for $u=u(M,\tau')=\lambda(\tau')$ and $a^2=2 |\tau|/(1+\delta)$, we have
  \[
  Y_a( \lambda(\tau')) \geq \frac {1+\delta}{2 |\tau|} \, ( \frac 2{\lambda(\tau')^2} -1 )
  \]
  and on the other hand by \eqref{eqn-YM1},
  \[
  Y(\lambda(\tau'),\tau') = Y(u(M, \tau'), \tau') = \frac 1{2|\tau'|} \big ( \frac 2{\lambda^2(\tau')} - 1 \big ) + \frac 1{4|\tau'|^2} + o_M(\frac 1{|\tau'|}).
  \]
  Hence, in order to guarantee that $Y(\lambda(\tau'), \tau') \leq Y_a ( \lambda(\tau') )$, it is sufficient have
  \[
  \frac 1{2|\tau'|} \big ( \frac 2{\lambda^2(\tau')} - 1 \big ) + \frac 1{2|\tau'|^2} \leq \frac {1+\delta}{2 |\tau|} \, ( \frac 2{\lambda^2(\tau')} -1 )
  \]
  for all $\tau' \leq \tau$.
  But since $\lambda(\tau')=u(M,\tau')$, \eqref{eqn-u2-3} gives
  \[
  \frac 2{\lambda^2(\tau')} - 1 = \frac{M^2-2}{4|\tau'|} + o_M( \frac 1{|\tau'|}).
  \]
  Using the above and that $|\tau'| > |\tau|$, we conclude that it is sufficient to have
  \[
  \frac {M^2-2}{ 4 |\tau'|^2} + \frac 1{2|\tau'|^2} + o_M(\frac 1{|\tau'|^2}) \leq \frac{M^2-2}{4 |\tau'|^2 } (1+\delta) + o_M ( \frac 1{ |\tau'|^2 }),
  \qquad \forall \tau'\leq \tau\leq \tau_0
  \]
  which is guaranteed, after we absorb lower order terms, if
  \[
  \frac{M^2-2}{ 4 |\tau'|^2} + \frac 1{|\tau'|^2} \leq \frac{M^2-2}{4 |\tau'|^2} (1+\delta),
  \]
  or equivalently,
  $M^2+2 \leq (M^2-2) \, (1+\delta) $.  Hence, it is sufficient to choose $M = M(\delta)$ such that
  $M^2+2 = (M^2-2) (1+\delta)$, implying that
  \[
  M = M_\delta := \sqrt{2 + \frac 4{\delta}}.
  \]
  For this choice of $M$ the above inequality is satisfied for all $\tau' \leq \tau \leq \tau_0(\delta,\theta)\ll -1$.
  In all the above calculations we used the estimate $o_M(1) \leq \frac 1{100}$, for $\tau' \leq \tau$.  This is possible with our choice of
  $ M_\delta$, as we can take $\tau_0=\tau_0(\delta)$ to be very negative.
  
  We will now show that for the same choice of $a = a(\tau)$, we have
  \[
  Y(a^{-1} r^*, \tau') \leq Y_a(a^{-1} r^*), \qquad \forall \, \tau' \leq \tau \leq \tau_0.
  \]
  Since our solution $Y$ satisfies $Y(u,\tau) \leq 1$ always, it is sufficient to show that $Y_a(a^{-1} r^*) \geq 1$.  This readily follows from the construction of $Y_a$ in \cite{Br}, where
  $Y_a(a^{-1} r^*, \tau') = 2 + \beta_a(r^*)$ and $1+\beta_a(r^*) \geq 1$ (see in \cite{Br}, Proposition 2.4 for the definition of the function $\beta_a$,
  and in the proof of Proposition 2.8 for the property $1+\beta_a(r^*) \geq 1$).
  
  Finally, lets show that for our choice of $a=a(\tau)$ we can chose $\bar \tau \ll \tau \leq \tau_0 \ll -1$, such that
  \begin{equation}\label{eqn-Y111}
    Y(u, \bar \tau) \leq Y_a(u), \qquad \mbox{for all} \,\, u \in [a^{-1} r^*, \lambda(\bar \tau)].
  \end{equation}
  To this end, we use Lemma 2.7 in \cite{Br} which implies in our rescaled variables that
  \[
  \liminf_{\bar \tau \to -\infty} \sup_{u \geq a^{-1} r^*} Y(u, \bar \tau) =0.
  \]
  On the other hand, the construction of the barrier $Y_a(u)$ in \cite{Br} implies that
  \[
  \inf_{u \geq a^{-1} r^*} Y_a(u) =c(\tau) >0
  \]
  which implies that \eqref{eqn-Y111} holds, for some $\bar \tau \ll \tau$.
  
  We can now apply the comparison principle between our solution $Y(u,\tau)$ of \eqref{eqn-Y} and the supersolution $Y_a(u) $
  of the same equation on the domain $Q(\tau)$ defined
  by \eqref{eqn-Qtau}, to conclude that $Y(u,\tau') \leq Y_a(u)$ on $Q(\tau)$.  This in particular holds at $\tau'=\tau$, namely we have
  \begin{equation}\label{eqn-Y13}
    Y(u,\tau) \leq Y_{a(\tau)}(u), \qquad \,\, \forall \,\, u \in [a(\tau)^{-1} r^*, \lambda(\tau)].
  \end{equation}
  
  To finish the proof of the Proposition, we will bound from above $Y_{a(\tau)}(u)$ for $u \in [\theta, \sqrt{2}]$.  Recall the definition of $Y_a(u)$ for
  $u \in [N\, a^{-1}, \sqrt{2}]$ as
  \[
  Y_a(u) = \varphi \big ( a\, \frac{u}{\sqrt{2} } \big ) - a^{-2} + a^{-4} \zeta(u)
  \]
  where $ \varphi(s) = s^{-2} + O(s^{-4})$, as $s \gg +\infty$ and $\zeta(s)$ bounded for $s \in [\theta, \sqrt{2}]$ with a bound depending on $\theta$.
  Recalling the definition of $a=a(\tau)$ to satisfy ${\ds a^{2} = \frac{2 |\tau|}{1+\delta}}$, we have $a \gg 1$ for $\tau \leq \tau_0 \ll -1$.  Hence,
  \begin{equation}\label{eqn-Y14}
    Y_a(u) \leq a^{-2} (2u^{-2} -1 ) + C_\theta \, a^{-4} \leq \frac{1+\delta}{|\tau|} (2u^{-2} -1 ) + \frac {C_\theta}{|\tau|^2}
  \end{equation}
  for all $ u \in [\theta, \sqrt{2}].$ Since $r^*\, a^{-1} < \theta$ and $\lambda(\tau)=u(M,\tau) < \sqrt{2}$ for $\tau \leq \tau_0 \ll -1$, we finally conclude by combining \eqref{eqn-Y13} and \eqref{eqn-Y14} that our desired bound \eqref{eqn-Y15} (with $\delta$ replaced by $2\delta$) holds on $u \in [\theta, \lambda(\tau)]$ finishing the proof of the proposition.
  
\end{proof}

\smallskip

We will now give the proof of our main result in this section, Proposition \ref{prop-intermediate}.

\begin{proof}[Proof of Proposition \ref{prop-intermediate}] We will show that the two bounds
  \begin{equation}\label{eqn-bound-ab}
    \liminf_{\tau \to +\infty} \bar u(z,\tau) \geq \sqrt{2 - \frac{z^2}{2}} \qquad \mbox{and} \qquad \limsup_{\tau \to +\infty} \bar u(z,\tau) \leq \sqrt{2 - \frac{z^2}{2}}.
  \end{equation}
  hold uniformly on $u \geq \theta$, for any $\theta \in (0, \sqrt{2})$.
  
  \smallskip
  \smallskip
  
  \noindent{\em The bound from below in \eqref{eqn-bound-ab}}:
  The desired bound
  clearly holds in the parabolic region $|\sigma| \leq M$, for any $M >0$ (see Lemma \ref{lemma-Y10}).
  Hence, it is sufficient to show that for the given $\theta \in (0, \sqrt{2})$ and for
  any $\delta >0$ small, there exists $\tau_0=\tau_0(\theta, \delta)$
  such that
  \begin{equation}\label{eqn-lower10}
    u(z,\tau) \geq \sqrt{2 - (1+\delta) \, \frac{z^2}{2}}
  \end{equation}
  holds on the $\cI_{M,\theta} \cap \{ z : \,\, (1+\delta)\, z^2 <4 \}$, for $M \gg 1$ (we need to have $M \geq M_\delta$,
  where $M_\delta$ is defined in Proposition \ref{prop-Y-bound-above}).
  This will follow by integrating on $ [M, \sigma]$, for $|\sigma | \geq M \gg 1$ the bound
  \[
  u_\sigma^2(\sigma,\tau) \leq \frac {1+\delta}{2|\tau|} \big ( 2\, u(\sigma,\tau)^{-2} - 1 \big ), \qquad \tau \leq \tau_0(\theta, \delta) \ll -1
  \]
  shown in Proposition \ref{prop-Y-bound-above}.  Assume that $\sigma \geq M$ as the case $\sigma < -M$ is similar.
  Multiplying both sides of the previous inequality by $u^2$ and taking square roots, we obtain
  (since $u_\sigma <0$ for $\sigma \geq M \gg 1$), the differential inequality
  \[
  - \frac{u u_\sigma}{ \sqrt{ 2- u^2} } \leq \frac {\sqrt{1+\delta}}{\sqrt{2|\tau|}}.
  \]
  Hence, after integration on $[M, \sigma]$ we get
  \[
  \sqrt{ 2- u^2(\sigma,\tau)} - \sqrt{ 2- u^2(M,\tau)} \leq \frac {\sqrt{1+\delta}}{\sqrt{2|\tau|}} (\sigma - M).
  \]
  Next we use the parabolic region asymptotics \eqref{eqn-u13} which give for $M \gg 1$
  \[
  \sqrt{ 2- u^2(M,\tau)} = \frac{\sqrt{M^2-2}}{\sqrt{2|\tau|}} + o_M \big ( \frac 1{\sqrt{|\tau|}} \big ) \leq \frac{M}{\sqrt{2|\tau|}}\, \sqrt{1+\delta}
  \]
  and combine it with the previous estimate to obtain the bound
  \[
  \sqrt{ 2- u^2(\sigma,\tau) } \leq \frac {\sqrt{1+\delta}}{\sqrt{2|\tau|}} (\sigma - M) + \frac{M}{\sqrt{2|\tau|}} = \sqrt{1+\delta} \, \frac{ \sigma}{\sqrt{2|\tau|}}.
  \]
  Recalling our notation $\bar u(z,\tau) = u(\sigma,\tau)$, with $z:=\sigma/\sqrt{|\tau|}$, the above estimate can be written after being squared as
  \begin{equation}
    \label{eq-help-bound-120}
    2- \bar u^2(z,\tau) \leq (1+\delta) \, \frac{z^2}{2} \implies \bar u(z,\tau) \geq \sqrt{ 2 - (1+\delta) \, \frac{z^2}{2}}.
  \end{equation}
  The same estimate holds for $z < 0$
  and the bound \eqref{eqn-lower10} is shown.  Note that the above inequality holds only when $(1+\delta)\, z^2 <4 $
  and for $\bar u \geq \theta$.  It follows that
  \[
  \liminf_{\tau \to -\infty} \bar u(z,\tau) \geq \sqrt{ 2 - (1+\delta) \, \frac{z^2}{2}}
  \]
  and since this holds for any $\delta >0$, we conclude the bound from below in \eqref{eqn-bound-ab}.
  
  \smallskip
  
  \noindent {\em The bound from above in \eqref{eqn-bound-ab}}: This estimate follows from our equation \eqref{eqn-uz}
  after we use the concavity of $\bar u$ and estimate \eqref{eqn-Y15} to deduce from \eqref{eqn-uz} the first order partial differential inequality
  \begin{equation}\label{eqn-buz-25} \bar u_\tau \leq - \frac{z}{2} \, \bar u_{z} \big (1 - \frac{\kappa}{|\tau|} \big ) - \frac 1{ \bar u} + \frac {\bar u}2
  \end{equation}
  which holds on $\bar u \geq \theta$, for any $\theta \in (0,\sqrt{2})$ and for some $\kappa = \kappa(\theta) >0$.  Let us first verify that \eqref{eqn-buz-25} follows from our equation \eqref{eqn-uz} and $\bar u_{zz} \leq 0$.  First, $\bar u_{zz} <0$ and $u_z <0$ iff $z >0$, implies that
  \[
  - \frac 1{|\tau|} \bar J(z,\tau) \, {\bar u}_{z} = - \frac {2\,\bar u_z} {|\tau|} \int_0^z \bar u_{zz}\, dx \leq 0.
  \]
  Next we observe that
  \[
  \frac{{\bar u}_{z}^2}{|\tau| u} \leq \frac z2 \, \frac{\kappa}{|\tau|} \qquad \mbox{or equivalently} \qquad \frac{u_\sigma^2}{u} \leq \frac z2 \, \frac{\kappa}{|\tau|}
  \]
  for some $\kappa = \kappa(\theta) >0$.
  Indeed, the latter estimate follows immediately by \eqref{eqn-Y15} and \eqref{eq-help-bound-120}.  Using once more $\bar u_{zz} < 0$, to estimate the diffusion term in \eqref{eqn-uz}, we conclude that \eqref{eqn-buz-25}
  holds.
  
  Now let's integrate \eqref{eqn-buz-25}.
  Defining $\bv := \bu ^2 - 2$, \eqref{eqn-buz-25} becomes equivalent to
  \[
  \bv_\tau \le -\frac z2\, \big (1 - \frac{\kappa}{|\tau|} \big )\, \bv_z + \bv.
  \]
  We see $\bv(z,\tau)$ is a subsolution to the first order partial differential equation
  \[
  \frac{\pd}{\pd\tau} w = -\frac z2 \, \big (1 - \frac{\kappa}{|\tau|} \big ) \, w_z + w,
  \]
  which we can write as
  \begin{equation}
    \label{eq-method-char}
    \frac{d}{d\tau} w(z(\tau),\tau) = w(z(\tau),\tau),
  \end{equation}
  where
  \begin{equation}
    \label{eq-characteristics}
    \frac{d}{d\tau} z = \frac z2\, \big (1 - \frac{\kappa}{|\tau|} \big )
  \end{equation}
  is the characteristic equation for \eqref{eq-method-char}
  (see Figure~\ref{fig-characteristics}).
  
  Assume the curve $(z(\tau),\tau)$ connects $(z,\tau)$ and $(z_1,\tau_1)$ with $z_1 = M/\sqrt{|\tau_1|}$, for $M > 0$ big.
  Integrate \eqref{eq-characteristics} from $\tau_1$ to $\tau$ to get
  \[
  \tau_1 = \tau + \log \Big(\frac{M^2}{z^2|\tau_1|}\Big) + \kappa \, \log \Big (\frac{|\tau|}{|\tau_1|}\Big )
  \]
  from which it also follows that for $\bar u \geq \theta$, we have
  \[
  \tau_1 = \tau \, \big ( 1 + \epsilon_1(\tau) \big )
  \]
  where $\epsilon_1(\tau) <0$ satisfies $\lim_{\tau \to -\infty} \epsilon_1(\tau)=0$.
  At the point $z_1 = \frac{M}{\sqrt{|\tau_1|}}$ we can use \eqref{eqn-u13} to compute $\bv$:
  \[
  \bv (z_1,\tau_1) = \bu(z_1, \tau_1)^2 - 2 = -\frac{M^2-2}{2|\tau_1|}(1+\epsilon(\tau_1))
  \]
  where here and below $\epsilon(\tau)$ will denote functions which may change from line to line but all satisfy $\lim_{\tau \to -\infty} \epsilon(\tau)=0$.
  \begin{figure}\centering
    \includegraphics[scale=0.8]{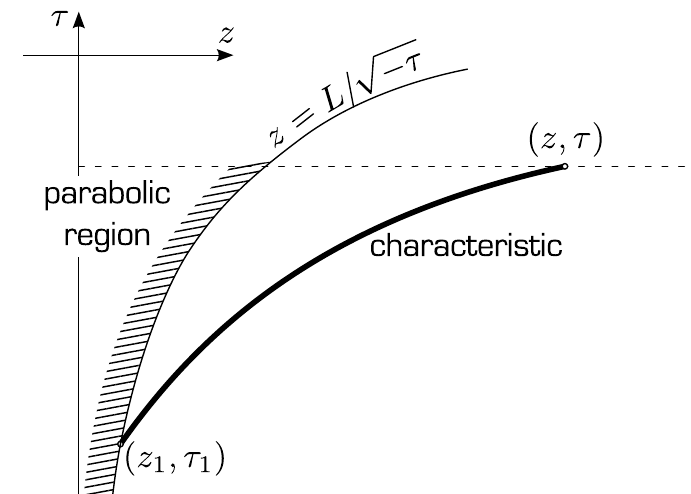}
    \caption{To estimate $\bv$ at $(z, \tau)$ we follow the characteristic through $(z,\tau)$ back to the boundary of the parabolic region, where $y=M$, $z=M/\sqrt{|\tau|}$.}
    \label{fig-characteristics}
  \end{figure}
  On the other hand, if we integrate \eqref{eq-method-char} from $\tau$ to $\tau_1$ we get
  \[
  w(z,\tau) = e^{\tau - \tau_1}\, w(z_1,\tau_1),
  \]
  and we can start $w$ with the initial condition $w(z_1,\tau_1) = \bv(z_1, \tau_1)$, so that
  \[
  w(z,\tau) = - \frac{z^2 |\tau|}{|\tau_1|}\,\frac{M^2-2}{2M^2} \, (1 + \epsilon(\tau_1)),
  \]
  where we have seen hat $ \tau_1 = \tau \big ( 1 + \epsilon_1(\tau) \big ).$
  Therefore,
  \[
  w(z,\tau)
  = -\frac{z^2|\tau|}{|\tau| \, \big (1 + \epsilon(\tau) \big ) }
  \frac{M^2-2}{2M^2}\,
  (1 + \epsilon_1(\tau)) = -z^2 \, \frac{M^2-2}{2M^2} \, \big (1+ \bar \epsilon(\tau) \big),
  \]
  where $\lim_{\tau \to + \infty } \bar \epsilon(\tau)=0$.
  Since $\bv (z_1,\tau_1) = w(z_1,\tau_1)$, by the maximum principle applied to \eqref{eq-method-char}, along characteristics $(z(\tau),\tau)$ connecting $(z_1,\tau_1)$ and $(z,\tau)$ we have
  \[
  \bv (z,\tau) \le w(z,\tau).
  \]
  This implies that for all $z\ge M/\sqrt{|\tau|}$ and as long as $\bar u(z,\tau) \geq \theta$ one has
  \[
  \bu (z,\tau) \le \sqrt{2 - \frac{M^2-2}{2M^2}\, z^2\, (1+ \epsilon(\tau))},
  \]
  where again $\lim_{\tau\to -\infty} \epsilon(\tau) = 0$.
  Hence, for all $z\in (0, \sqrt{2})$ with $\bar u(z,\tau) \geq \theta$, we have
  \[
  \limsup_{\tau\to-\infty} \bu(z, \tau) \le \sqrt{2 - \frac{M^2-2}{2M^2}\, z^2}.
  \]
  Since this holds for all $M>0$, we finally conclude the upper bound in \eqref{eqn-bound-ab}
  which holds uniformly on $\bar u \geq \theta$, for any $\theta \in (0, \sqrt{2})$.
  
\end{proof}

\section{Tip region asymptotics}

\label{sec-tip}
Recall that in commuting variables $\sigma$ and $\tau$ the rescaled radius $u(\sigma,\tau)$ satisfies equation \eqref{eqn-un}.  We will use this
equation to obtain an estimate for the rescaled diameter $\bd (\tau)$.  In our reflection symmetric case the location of the tips
is $\sigma = \sigma_\pm(\tau) = \pm \sigma(\tau)$, hence $\bd (\tau) = \sigma_+(\tau) - \sigma_-(\tau) = 2\sigma(\tau)$.
Since $u(\sigma(\tau),\tau) = 0$, after differentiating it in $\tau$ and using that $u_{\sigma}(\sigma(\tau),\tau) = -1$ and $u_{\sigma\sigma}(\sigma(\tau),\tau) = 0$, we obtain that $\bar \sigma(\tau)$ satisfies
\begin{equation}
  \label{ODE-diameter}
  \sigma'(\tau) = \frac{\sigma(\tau)}{2} + J(\sigma(\tau),\tau),
\end{equation}
where we recall that $J(\sigma,\tau)$ is the non-local term given by \eqref{eqn-defn-J}.
Using results in \cite{Br} we prove the following Proposition.

\begin{proposition}
  \label{prop-J}
  For every $\epsilon > 0$, there exists a $\tau_0 = \tau_0(\epsilon) \ll -1$ so that for all $\tau \le \tau_0$ we have
  \[
  \left| \frac{J(\sigma(\tau),\tau)}{\sqrt{\kappa(\tau)}} - C_0\right| < \epsilon, \qquad \mbox{with}\,\,\, C_0 := \int_0^{\infty} \frac{Z_0'(\rho)}{\rho \sqrt{Z_0(\rho)}}
  \, d\rho
  \]
  where $Z_0(\rho)$ is the Bryant soliton whose maximal curvature is equal to one, and $\kappa(\tau)$ is the maximum of scalar curvature of our evolving metric at time $\tau$.
\end{proposition}

\begin{proof}
  By Corollary \ref{cor-max-curv-location} we know the maximum of scalar curvature is attained at $\sigma = \sigma(\tau)$ and $\sigma = -\sigma(\tau)$.  Denote by $p_{\tau}$ the point of the maximum of scalar curvature at time $\tau$.  Take any sequence $\tau_i\to -\infty$ and set $p_i:=p_{\tau_i}$.  Dilate our metric by $\kappa_i = \kappa(\tau_i) := R(p_i,\tau_i)$.  By the compactness argument of Perelman for $\kappa$-solutions (see section 11 in \cite{Pe1}) we can extract a subsequence of dilated solutions converging to a $\kappa$-solution.  We claim that this limiting $\kappa$-solution needs to be non compact.  To show this, lets argue by contradiction and assume that the limit is compact.  This would imply there existed a uniform constant $C$ so that $\sqrt{R(p_i,\tau_i)}\, \sigma(\tau_i) \le C$ for all $i$.  By Proposition \ref{prop-intermediate} we have $\sigma(\tau_i) = 2\sqrt{|\tau_i|}\, (1 + o(1))$ and hence we would have $R(p_i,\tau_i) \le \frac{C}{|\tau_i|}$ for all $i$.  Since $p_i$ is the point of maximal scalar curvature at time $\tau_i$, we would in particular have the bound $R(0,\tau_i) \le \frac{C}{|\tau_i|}$, for all $i$.  Since $R(0,\tau_i) \ge \frac{1}{u(0,\tau_i)}$ and the $\lim_{i\to\infty} u(0,\tau_i) = \sqrt{2}$ we immediately obtain contradiction.
  
  Once we know the limiting $\kappa$-solution is non compact, by the result in \cite{Br} we know that the limit is isometric to either the shrinking cylinder or the Bryant soliton whose maximal curvature is one.  Since at the tips the two curvatures $K_0$ and $K_1$ are equal we conclude that the limit must be the Bryant soliton of maximal curvature one.  This is equivalent to saying that if
  \[
  Z(\rho,\tau) = Y(u,\tau), \qquad \rho = \sqrt{\kappa(\tau)}\, u,
  \]
  then the $\lim_{\tau\to -\infty} Z(\rho,\tau) = Z_0(\rho)$, where $Z_0(\rho)$ is the unique Bryant soliton with maximal curvature one, namely $Z_0(\rho)$ is a solution of the equation
  \begin{equation}\label{eqn-Z0-in}
    Z_0 \, Z_{0\rho\rho} - \frac 12\, (Z_{0\rho})^2 + (1 - Z_0)\, \frac{Z_{0\rho}}{\rho} + 2\, (1-Z_0) \, \frac{Z_0}{\rho^2}=0
  \end{equation}
  satisfying $Z_0(0)=1$.
  
  It was shown by R.  Bryant in \cite{Bry} that for large and small $\rho$ the function $Z_0(\rho)$ satisfies the asymptotics
  \begin{equation}
    \label{eq-Z0-asymptotics}
    Z_0(\rho) =
    \begin{cases}
      \rho^{-2} + O(\rho^{-4}), &\mbox{as} \,\,\, \rho\to\infty \
      \\[2pt]
      1 - \rho^2 /6 + O(\rho^4), &\mbox{as}\,\,\, \rho\to0
    \end{cases}
  \end{equation}
  (see also in Appendix \ref{appendix-bryant} for the derivation of the exact constants).
  
  For any $\tau \ll -1$, define $\sigma_L(\tau) >0$ to be the number at which
  ${\ds u(\sigma_L(\tau), \tau) = \tfrac{L}{\sqrt{\kappa(\tau)}}}$ where
  $\kappa(\tau)=R(\sigma(\tau), \tau)$ is the scalar curvature at the tip.
  Then, we split
  \[
  J(\sigma(\tau),\tau) = - 2 \int_0^{\sigma(\tau)} \frac{u_{\sigma\sigma}}{u} d\sigma' = - 2 \int_0^{\sigma_L(\tau)} \frac{u_{\sigma\sigma}}{u} d\sigma'
  -2 \int_{\sigma_L(\tau)}^{\sigma(\tau)} \frac{u_{\sigma\sigma}}{u} d\sigma'.
  \]
  By integration by parts we have
  \[
  0 \leq - \int_0^{\sigma_L(\tau)} \frac{u_{\sigma\sigma}}u\, d\sigma' = -\int_0^{\sigma_L(\tau)} \frac{u_{\sigma}^2}{u^2}\, d\sigma' - \frac{u_\sigma}{u} |_{u = \tfrac{L}{\sqrt{\kappa(\tau)}}} \leq - \frac{u_\sigma}{u} |_{u = \tfrac{L}{\sqrt{\kappa(\tau)}}}.
  \]
  On the other hand the convergence to the soliton implies that
  \[
  - \frac {u_\sigma}{u} |_{u = \tfrac{L}{\sqrt{\kappa(\tau)}}} = \frac{\sqrt{\kappa(\tau)}}{L} \,
  \sqrt{ Y( \frac{L}{\sqrt{\kappa(\tau)}} ,\tau)} \approx \frac{\sqrt{\kappa(\tau)}}{L} \, \sqrt{Z_0(L) }
  \leq C \, \frac{\sqrt{\kappa(\tau)}}{L^2},
  \]
  where in the last inequality we used the soliton asymptotics \eqref{eq-Z0-asymptotics}.
  Therefore we conclude that
  \[
  - \int_0^{\sigma_L(\tau)} \frac{u_{\sigma\sigma}}u\, d\sigma' \leq C \, \frac{\sqrt{\kappa(\tau)}}{L^2}.
  \]
  For the other integral, using the change of variables $Y(u,\tau) = u_\sigma^2(\sigma,\tau)$, $u=u(\sigma,\tau)$ and the convergence to
  the soliton $Z_0$ we have
  \[
  -2 \int_{\sigma_L(\tau)}^{\bdt} \frac{u_{\sigma\sigma}}{u} d\sigma' = \int_0^{\tfrac{L}{\sqrt{\kappa(\tau)}}} \frac{Y_u}{u\sqrt{Y}}\, du
  \approx \sqrt{\kappa(\tau)} \, \int_0^{L} \frac{Z_0'(\rho)}{\rho \sqrt{Z_0(\rho)}}\, d\rho
  \]
  where we used that
  $Y_u = Z_0'(\rho)\, \rho_u = \sqrt{\kappa(\tau)} \, Z_0'(\rho)$ and ${\ds \frac{du}{u} = \frac{d\rho}{\rho}}.$
  Combining the above yields that for any $L \gg 1$ we have
  \[
  \Big | \frac{J(\sigma(\tau),\tau)}{\sqrt{\kappa(\tau)}} - C_0 \Big |
  \leq \frac{C}{L^2} + \Big | \int_0^{L} \frac{Z_0'(\rho)}{\rho \sqrt{Z_0(\rho)}}\, d\rho - C_0 \Big |
  \]
  and since $C_0 := \int_0^{\infty} \frac{Z_0'(\rho)}{\rho \sqrt{Z_0(\rho)}}\, d\rho$, we obtain
  \[
  \Big | \frac{J(\sigma(\tau),\tau)}{\sqrt{\kappa(\tau)}} - C_0 \Big | \leq \frac{C}{L^2} + \int_L^{\infty} \frac{Z_0'(\rho)}{\rho \sqrt{Z_0(\rho)}}\, d\rho
  \to 0, \quad \, \mbox{as} \,\, L \to +\infty.
  \]
  This concludes the proof of the proposition.
  
\end{proof}

The next Proposition relates the constant $C_0$ with the asymptotic behavior of the Bryant soliton $Z_0(\rho)$ of maximum scalar curvature one, as $\rho \to -\infty$.

\begin{proposition}\label{prop-C0-11} Assume that $Z_0$ is the Bryant soliton with maximum scalar curvature at the tip equal to one.  Then,
  \[
  C_0 := \int_0^{\infty} \frac{Z_0'}{\rho \sqrt{Z_0}} \, d\rho = -1.
  \]
\end{proposition}

\begin{proof}
  
  We first observe that
  \[
  C_0 := \int_0^{\infty} \frac{Z_0'}{\rho \sqrt{Z_0}} \, d\rho = \int_0^{\infty} \frac{2}{\rho} \, \frac{d \sqrt{Z_0}}{d\rho} \, d \rho.
  \]
  The idea is to express the integrand ${\ds \frac{2}{\rho} \, \frac{d}{d\rho} \sqrt{Z_0(\rho)} }$ in divergence form, using the soliton equation \eqref{eqn-Z0-in}.
  This will allow us to integrate it.
  
  We observe that \eqref{eqn-Z0-in} can be recognized as the pressure equation of $\Psi:=\sqrt{Z_0}$.  Indeed, a direct calculation shows that if $Z_0= \Psi^2$, then
  \[
  Z_0 \, Z_{0\rho\rho} - \frac 12\, (Z_{0\rho} )^2 = \Psi^2 \big (2 \Psi \Psi_{\rho\rho} + 2 \Psi_\rho^2 \big ) - \frac 12 (2 \Psi \Psi_{\rho} )^2 = 2 \Psi^3 \Psi_{\rho\rho}
  \]
  hence the soliton equation can be expressed in terms of $\Psi$ as
  \[
  2 \Psi^3 \Psi_{\rho\rho} + 2 \Psi \, (1 - \Psi^2)\, \frac{\Psi_\rho}{\rho} + 2\, (1-\Psi^2) \, \frac{\Psi^2}{\rho^2}=0.
  \]
  Dividing by $2\Psi^3$ gives
  \begin{equation}\label{eqn-PSI-C} \Psi_{\rho\rho} + (\Psi^{-2} - 1)\, \frac{\Psi_{\rho}}{\rho} + \frac 1{\rho^2}\, (\Psi^{-1}-\Psi)=0.
  \end{equation}
  Next notice that
  \[
  \frac d{d\rho} \big ( \frac 1\rho \, (\Psi-\Psi^{-1}) \big ) = - \frac 1{\rho^2} \, (\Psi-\Psi^{-1}) +(\Psi^{-2} + 1)\, \frac{\Psi_{\rho}}{\rho}
  \]
  which gives
  \[
  \frac 1{\rho^2}\, (\Psi^{-1}-\Psi) = \frac d{d\rho} \big ( \frac 1\rho \, (\Psi-\Psi^{-1}) \big ) - (\Psi^{-2} + 1)\, \frac{\Psi_{\rho}}{\rho}.
  \]
  Substituting this into \eqref{eqn-PSI-C} yields that the soliton equation can be expressed in terms of $\Psi$ as
  \[
  \Psi_{\rho\rho} + \frac d{d\rho} \big ( \frac 1\rho \, (\Psi-\Psi^{-1}) \big ) - 2 \frac{\Psi_\rho}{\rho} =0
  \]
  implying that
  \begin{equation}\label{eqn-divergence}
    \frac 2\rho \, \Psi_\rho = \frac d{d\rho} \big ( \Psi_\rho + \frac 1\rho \, (\Psi-\Psi^{-1}) \big ).
  \end{equation}
  
  \sk
  Now using \eqref{eqn-divergence} we obtain
  \[
  \begin{split}
    C_0 &:= \int_0^{\infty} \frac{2}{\rho} \, \frac{d \sqrt{Z_0}}{d\rho} \, d \rho = \int_0^{\infty} \frac 2\rho \, \Psi_\rho \, d \rho = \int_0^{\infty} \frac d{d\rho} \big ( \Psi_\rho + \frac 1\rho \, (\Psi-\Psi^{-1}) \big )\, d\rho \\
    & = \lim_{\rho \to +\infty } \big ( \Psi_\rho + \frac 1\rho \, (\Psi-\Psi^{-1}) \big ) - \lim_{r\to 0} \big ( \Psi_\rho + \frac 1\rho \, (\Psi-\Psi^{-1}) \big ).
  \end{split}
  \]
  By the Bryant soliton asymptotics \eqref{eq-Z0-asymptotics} it is easy to see that
  \[
  \lim_{r\to 0} \big ( \Psi_\rho + \frac 1\rho \, (\Psi-\Psi^{-1}) \big ) =0, \quad \lim_{r\to +\infty } \big ( \Psi_\rho + \frac 1\rho \, (\Psi-\Psi^{-1}) \big ) = - 1
  \]
  thus concluding that
  $C_0:= -1$, as claimed.
\end{proof}

\sk
As a corollary of the previous two Propositions we next compute the maximum rescaled scalar curvature.

\begin{corollary}
  \label{cor-curv-asymp} The maximum rescaled scalar curvature $\kappa(\tau)$ of our solution satisfies the asymptotic behavior
\begin{equation} \label{eqn-curv-asymp} \kappa(\tau) = |\tau| (1 + o(1)), \qquad \mbox{as} \,\, \tau\to -\infty.\end{equation}
  
\end{corollary}

\begin{proof}
  On the one hand by Propositions \ref{prop-J} and \ref{prop-C0-11} we have
  \[
  J(\sigma(\tau),\tau) = - \sqrt{\kappa(\tau)} \, (1 + o(1)).
  \]
  On the other hand, the intermediate region asymptotics discussed in section \ref{section-intermediate} imply that
  \begin{equation}
    \label{eq-diameter-asymp}
    \sigma(\tau) = 2\, \sqrt{|\tau|}\, (1 + o(1)),\qquad \mbox{as} \,\,\, \tau\to -\infty.
  \end{equation}
  Hence by \eqref{ODE-diameter}
  \[
  \frac{\sigma'(\tau)}{\sigma(\tau)} = \frac12 + \frac{J(\sigma(\tau),\tau)}{\sigma(\tau)} = \frac12 - \frac{\sqrt{\kappa(\tau)}\, (1+o(1))}{2 \sqrt{|\tau|}}.
  \]
  Integrating this from $\tau$ to $\tau+\epsilon$, for any small $\epsilon > 0$ yields
  \[
  \ln\frac{\sigma(\tau+ \epsilon)}{\sigma(\tau)} = \frac{\epsilon}{2} - \frac{ (1 + o(1))}{2 \sqrt{|\tau|}}\, \int_{\tau}^{\tau+\epsilon} \sqrt{\kappa(\tau')}\, d\tau'.
  \]
  Using also \eqref{eq-diameter-asymp} we obtain
  \[
  -\frac{ \epsilon}{2|\tau|}\, (1+o(1)) = \frac{\epsilon}{2} - \frac{(1 + o(1))}{2\sqrt{|\tau|}}\, \int_{\tau}^{\tau+\epsilon} \sqrt{\kappa(\tau')}\, d\tau'.
  \]
  Divide by $\epsilon$ and then let $\epsilon\to 0$ above, to conclude that \eqref{eqn-curv-asymp} holds.
\end{proof}

Using the above result and similar arguments as in the proof of Theorem 1.6 in \cite{ADS1}, we obtain the following convergence theorem which we express in
terms of the unrescaled solution $(M, g(t))$.

\begin{theorem}
  Denote by $p_t$ the point of maximal scalar curvature of $(M,g(t))$.  For any $t < 0$ define the rescaled metric
  \[
  \tilde{g}_{t}(\cdot,s) = k(t) \, g(\cdot,t+k(t)^{-1}\, s),
  \]
  where $k(t) = R_{\max}(t) = R(p_t,t).$ Then
  \[
  k(t) = (1 + o(1))\, \frac{\log|t|}{|t|}, \qquad \mbox{as} \,\, t \to -\infty.
  \]
  Moreover, the family of rescaled solutions $(M,\tilde{g}_{t}(\cdot,s), p_{t})$ to the Ricci flow, converges to the unique Bryant soliton of maximal scalar curvature equal to one,
  namely the unique rotationally symmetric steady soliton with maximal scalar curvature equal to one.
\end{theorem}

\begin{proof}
  The proof follows by the same arguments as in the proof of Theorem 1.6 in \cite{ADS1}, using Corollary \ref{cor-curv-asymp} and Hamilton's Harnack estimate for the Ricci flow shown in \cite{Ha1}.
\end{proof}

\appendix

\section{Bryant soliton}

\label{appendix-bryant}

In \cite{Bry} Bryant showed that up to constant multiples, there is only one complete, steady, rotationally symmetric soliton in dimension three that is not flat.  It has positive sectional curvature which reaches the maximum at the center of rotation.  He showed that if we write the complete soliton (whose maximal sectional curvature is $1/6$) in the form $g = d\bar{s}^2 + f(\bar{s})^2 g_{S^2}$, where $\bar{s}$ is the distance from the center of rotation, one has, for large $\bar{s}$ the following asymptotics: the aperature $a(\bar{s})$ has leading order term $\sqrt{2\bar{s}}$, the radial sectional curvature has leading order term $1/4\bar{s}^2$ and the orbital sectional curvature has leading order term $1/{2\bar{s}}$.  Note that in our old notation we have $s = s(t) - \bar{s}$, where $s(t)$ is the distance from the origin to the center of rotation (or the point of maximal scalar curvature).

As it is explained at the end of Section \ref{sec-equations}, around the point of maximal scalar curvature (which in this case coincides with the center of rotation, see Corollary \ref{cor-max-curv-location}), it is more convenient to write the metric in the form $f_s^{-2} \, df^2+ f^2 g_{S^2}$, where $Z = f_s^2$ is a function of $f$.  In the case of a steady Ricci soliton, the profile function $Z$ is known to satisfy the equation
\[
Z Z_{ff} - \frac12 Z_f^2 + (1 -Z) \frac{Z_f}{f} + \frac{2(1-Z) Z}{f^2} = 0.
\]
Moreover, it is known that $Z(f)$ has the following asymptotics.  Near $ f = 0$, $Z$ is smooth and has the asymptotic expansion
\begin{equation}
  \label{eq-as-origin}
  Z(f) = 1 + b_0 \, f^2 + \frac25 \, b_0^2 \, f^4 + o(f^4),
\end{equation}
where $b_0 < 0$ is arbitrary.  As, $f \to + \infty$, $Z$ is smooth and has the asymptotic expansion
\begin{equation}
  \label{eq-as-infty}
  Z(f) = c_0 \, f^{-2} + 2 c_0^2 \, f^{-4} + o(f^{-4}),
\end{equation}
where $c_0 > 0$ depends on $b_0$.

\sk
\sk

We will next find (for the convenience of the reader) the exact values of the constants $b_0$ and $c_0$ in the above asymptotics
for the {\em Bryant soliton of maximal scalar curvature one}.  These exact asymptotics were used in Section \ref{sec-tip}.
\sk

Recall that the scalar curvature $R = 4K_0 + 2K_1$.  By Corollary \ref{cor-max-curv-location} we know that the maximal scalar curvature is attained at $\bar{s} = 0$, at which $K_0 = K_1$ and hence $R = 6 K_0$.  An easy consequence of Lemma \ref{lemma-bound-Q} is that the maximal scalar curvature being equal to one is equivalent to the maximal sectional curvature being equal to $1/6$.  Bryant's asymptotics
imply that for $\bar{s}$ is sufficiently large, the aperture $f \approx \sqrt{2\bar{s}}$, implying that $\bar{s} \approx {f^2}/{2}$.  We also have the radial sectional curvature $K_0 \approx 1/4\bar{s}^2 \approx f^{-4}$, for $a \gg 1$.  On the other hand, using \eqref{eq-as-infty} we also have
\[
K_0 = -\frac{Z_f}{2f} \approx \frac{c_0}{f^4}, \qquad \mbox{for}\,\,\, a\gg1
\]
implying that $c_0 = 1$.
We also have at $f = 0$ (which corresponds to a point of maximal sectional curvature),
$1/6 = K_1 =(1-Z)/a^2 \approx -b_0$, for $a \ll1$,
implying that $b_0 = -1/6$.

\sk

Summarizing the above discussion we conclude the following asymptotics for the {\em Bryant soliton with maximal scalar curvature equal to
one}:
\begin{equation}\label{eqn-Zasym10}
  Z(f)= \begin{cases} 1 - f^2/6 + O(f^4), &\qquad \mbox{as} \,\,\,\,\, f\to 0 \\
    f^{-2 } + O(f^{-4}), &\qquad \mbox{as} \,\,\,\, f\to\infty.
  \end{cases}
\end{equation}

\sk
\sk

\section{Properties of a rotationally symmetric solution}

\label{sec-apriori}

We assume throughout this section that $\psi(\cdot, s)$ is a solution to equation \eqref{eq-psi} or equivalently
$u(\cdot, \tau)$ is a solution of \eqref{eq-u}.  We will prove some crucial geometric properties which we used
in the proof of our main result in this paper.  We remark that we do not need to assume reflection symmetry.  Let us recall that we have denoted by $s_\pm(t)$ the tips of
our solution $M_t$ which in rescaled variables are denoted by $\sigma_\pm(\tau)$.  We will analyze our solution for $s\in [0,s_+(t)]$ (or $\sigma\in [0,\sigma_+(\tau)]$), since the solution for $s\in [s_-(t),0]$ (or $\sigma\in [\sigma_-(\tau),0]$) can be treated similarly.

We define next the scaling invariant quantity
\[
Q:= \frac{K_0}{K_1} = -\frac{\psi\psi_{ss}}{1 - \psi_s^2} = -\frac{u u_{\sigma\sigma}}{1 - u_{\sigma}^2}.
\]
In this section we will prove $Q \le 1$ everywhere on our manifold and that the maximum scalar curvature is attained at the tips $\sigma = \sigma_\pm(\tau)$.
A direct calculation shows that $Q$ satisfies the equation
\begin{equation}
  \label{eqn-Q}
  Q_t = Q_{ss} - \frac{2\psi_s}{\psi} (1-2Q) \, Q_s + \frac{2(1-Q)}{\psi^2} \, \left( (1-\psi_s^2) \, Q^2 + \psi_s^2 \, (2Q + 1)\right).
\end{equation}

Note that the closing conditions \eqref{closing-psi} imply that we can apply l'Hospital's rule to conclude
\[
K_1(s_+,t) = \lim_{s\to s_+} \left(-\frac{\psi_{ss}}{\psi}\right) = K_0(s_+,t)
\]
for all $t\in (-\infty,T)$.  Similarly we have $K_1(s_-,t) = K(s_-,t)$ for all $t\in (-\infty,T)$.
Hence, $Q(s_-,t) = Q(s_+,t) = 1$, for all $t\in (-\infty,T)$.

\begin{lemma}
  \label{lemma-bound-Q}
  Let $(S^3, g(t))$ be any rotationally symmetric compact, $\kappa$-noncollapsed ancient solution to the Ricci flow on $S^3$.
  Then, $Q \le 1$ on $S^3$, for all $t\in (-\infty,T)$.
\end{lemma}

\begin{proof}
  If $(S^3, g(t))$ is isometric to the sphere, then our claim clearly holds.  Hence, we may assume that the rescaled backward limit of our solution is the cylinder.
  
  \sk
  
  We have just observed that at both poles $s_\pm$ we have $Q = 1$.  Hence, $Q_{\max}(t)$ is achieved on the surface
  and $Q_{\max}(t) \geq 1$.  We will now show that $Q_{\max}(t) \leq 1$.
  At a maximum of $Q$ we have $Q_{s} = 0$ and $Q_{ss} \le 0$ so that
  \[
  \frac{dQ_{\max}}{dt} \le - \frac{2}{\psi^2}(Q_{\max} - 1)\, \left((1-\psi_s^2) Q_{\max}^2 + \psi_s^2 (2Q_{\max} + 1)\right).
  \]
  In particular this shows that if $Q_{\max}(t_0) > 1$ for some $t_0 < T$, then the same holds for all $t \le t_0$.
  
  Since the rescaled solution $u(\cdot, \tau)$ converges to $\sqrt{2(n-1)}$ as $\tau\to-\infty$ and it is concave,
  we have $u(\cdot, \tau) \leq C$, for all $\tau\leq \tau_0$.  Hence, $\psi(\cdot, t)\leq C\, \sqrt{T-t}$, for all $t\leq t_0 \ll -1$,
  implying that $Q_{\max}$ as a function of $\tau$, satisfies
  \[
  \frac{d(Q_{\max}-1)}{d\tau} \leq -c \, (Q_{\max} - 1) \left((1-u_\sigma^2) Q_{\max}^2 + u_\sigma^2 (2Q_{\max} + 1)\right),
  \]
  where we have used that $\psi_s=u_\sigma$.
  If $Q_{\max}$ is attained at a point where $u_{\sigma}^2 \le 1/2$, then using that $Q_{\max} > 1$
  \[
  \frac{d}{d\tau} (Q_{\max} - 1) \le -c \, (Q_{\max} - 1) \, Q_{\max}^2 \le -c \, (Q_{\max} - 1)^2,
  \]
  where constant $c$ may change from line to line, but is uniform, independent of $\tau$.  On the other hand, if $Q_{\max}$ is attained at a point where $\frac12 \le u_{\sigma}^2 \le 1$, then
  \begin{equation}
    \label{eq-always}
    \frac{d}{d\tau} (Q_{\max} - 1) \le -c \, (Q_{\max} - 1) (2Q_{\max} +1) \le - c\, (Q_{\max} - 1)^2.
  \end{equation}
  This means that we actually have \eqref{eq-always} holding for all $\tau \le \tau_0$, and thus
  \[
  \frac{d}{d\tau} \Big ( \frac{1}{Q_{\max} - 1} \Big ) \ge c, \qquad \mbox{ for all} \,\,\, \tau \leq \tau_0 \ll -1.
  \]
  As $\tau\to -\infty$ this leads to contradiction with $Q_{\max} - 1 > 0$.
\end{proof}

\begin{corollary}
  \label{cor-max-curv-location}
  For each $\tau$, the rescaled scalar curvature $R(\cdot,\tau)$ achieves its maximum at $\sigma = \pm \sigma(\tau)$.
\end{corollary}

\begin{proof}
  By Lemma \ref{lemma-bound-Q} we have that ${K_0}/{K_1} \le 1$ on $S^3$, for all $\tau$.  This implies
  \[
  u u_{\sigma\sigma} - u_{\sigma}^2 + 1 \ge 0, \qquad \mbox{which is equivalent to} \qquad (K_1)_{\sigma} \ge 0,
  \]
  for $\sigma \ge 0$, since in that region $u_{\sigma} \le 0$.  Using that observation we conclude that for the scalar curvature, at any point $0 \le \sigma < \sigma_+$ and any $\tau$, we have
  \[
  R(\sigma,\tau) = 4K_0(\sigma,\tau) + 2 K_1(\sigma,\tau) \le 6K_1(\sigma,\tau) \le 6K_1(\sigma_+,\tau) = R(\sigma_+,\tau).
  \]
  The case $\sigma <0$ can be shown similarly.
\end{proof}

\end{document}